\documentclass[12pt]{amsart}

\usepackage{amsfonts,amsthm,amsmath,amssymb,amscd,mathrsfs,tikz}
\usepackage{graphics}
\usepackage{indentfirst}
\usepackage{cite}
\usepackage{latexsym}
\usepackage[dvips]{epsfig}

\usepackage{color}

\newtheorem{theorem}{Theorem}[section]
\newtheorem{remark}{Remark}[section]
\newtheorem{definition}{Definition}[section]
\newtheorem{lemma}[theorem]{Lemma}
\newtheorem{pro}[theorem]{Proposition}

\renewcommand{\div}{{\rm div }}

\newcommand{\bt}{\begin{theorem}}
\newcommand{\bl}{\begin{lemma}}
\newcommand{\el}{\end{lemma}}
\newcommand{\et}{\end{theorem}}

\newcommand{\bn}{\begin{eqnarray}}
\newcommand{\en}{\end{eqnarray}}
\newcommand{\bnn}{\begin{eqnarray*}}
\newcommand{\enn}{\end{eqnarray*}}

\newcommand{\ba}{\begin{aligned}}
\newcommand{\ea}{\end{aligned}}
\newcommand{\be}{\begin{equation}}
\newcommand{\ee}{\end{equation}}

\newcommand{\R}{\mathbb{R}}

\newcommand{\Bx}{{\boldsymbol{x}}}

\newcommand{\Bv}{{\boldsymbol{v}}}
\newcommand{\Bnu}{{\boldsymbol{\nu}}}

\newcommand{\Bn}{{\boldsymbol{n}}}
\newcommand{\Bt}{{\boldsymbol{\tau}}}
\newcommand{\Bu}{{\boldsymbol{u}}}

\newcommand{\Be}{{\boldsymbol{e}}}

\newcommand{\BU}{\boldsymbol{U}}

\newcommand{\Bg}{{\boldsymbol{g}}}
\newcommand{\Bw}{{\boldsymbol{w}}}
\newcommand{\BD}{{\boldsymbol{D}}}
\newcommand{\Ba}{{\boldsymbol{a}}}

\newcommand{\Bp}{{\boldsymbol{\phi}}}
\newcommand{\Bh}{{\boldsymbol{h}}}

\newcommand{\mcH}{\mathcal{H}}

\newcommand{\mfc}{\mathfrak{c}}
\newcommand{\mfD}{\mathfrak{D}}

\begin{document}

\title[Two-dimensional flows with slip boundary condition]
{On the Leray problem for steady flows in two-dimensional infinitely long channels with slip boundary conditions}

\author{Kaijian Sha}
\address{School of mathematical Sciences, Shanghai Jiao Tong University, 800 Dongchuan Road, Shanghai, China}
\email{kjsha11@sjtu.edu.cn}

\author{Yun Wang}
\address{School of Mathematical Sciences, Center for dynamical systems and differential equations, Soochow University, Suzhou, China}
\email{ywang3@suda.edu.cn}

\author{Chunjing Xie}
\address{School of mathematical Sciences, Institute of Natural Sciences,
Ministry of Education Key Laboratory of Scientific and Engineering Computing,
and CMA-Shanghai, Shanghai Jiao Tong University, 800 Dongchuan Road, Shanghai, China}
\email{cjxie@sjtu.edu.cn}

\begin{abstract}
In this paper, we investigate the Leray problem for steady Navier-Stokes system with full slip boundary conditions in a two-dimensional channel with straight outlets. The existence of solutions with arbitrary flux in a general channel supplemented with slip boundary conditions, which tend to the associated shear flows at far fields, is established. Furthermore, if the flux is suitably small, the solution is proved to be unique. One of the crucial ingredients is to construct an appropriate flux carrier and to show a Hardy type inequality for flows with full slip boundary conditions.
\end{abstract}


\thanks{Updated on \today}

\maketitle

\section{Introduction}
An interesting and important problem in mathematical fluid mechanics is to study the solutions of the steady Navier-Stokes system
\begin{equation}\label{NS}
\left\{
\begin{aligned}
&-\Delta \Bu+\Bu\cdot \nabla \Bu +\nabla p=0 ~~~~&\text{ in }\Omega,\\
&{\rm div}~\Bu=0&\text{ in }\Omega,
\end{aligned}\right.
\end{equation}
in a channel domain $\Omega$, where the unknown function $\Bu=(u_1,\cdots,u_N)$ is the velocity and $p$ is the pressure. If the boundary condition $\Bu \cdot \Bn =0$ is prescribed, then the flux
\[
\Phi=\int_{\Sigma} \Bu \cdot \Bnu \,ds
\]
is a conserved quantity along each cross section $\Sigma$ of the channel, where $\Bnu$ is the unit normal of  $\Sigma$ pointing to the same direction.
 If $\Omega$ is a channel type domain with straight outlets at far fields, in 1950s, Leray proposed the problem to look for solutions for Navier-Stokes system \eqref{NS} with no-slip boundary conditions under the constraint that
\begin{equation}\label{far field}
	\Bu\to \BU \ \ \ \ \ \text{at far fields,}
\end{equation}
 where $\BU$ is the shear flow solution of Navier-Stokes system with flux $\Phi$ in the corresponding straight channel with no-slip boundary conditions. The problem is called Leray problem nowadays.  Without loss of generality, the flux $\Phi$ is always assumed to be nonnegative in this paper.

The major breakthrough on the Leray problem in infinitely long channels was made by Amick \cite{A1,A2,AF}, Ladyzhenskaya and Solonnikov \cite{LS}. It was proved in \cite{A1, LS} that Leray problem in a channel is solvable as long as the flux is small. Actually, the existence of solutions with arbitrary flux  was also proved in \cite{LS}. However, the far field behavior and  uniqueness of such solutions are not clear when the flux is large. The far field behavior of solutions was studied in \cite{AF}. One can refer to \cite{Ga,KP,NP1,NP2} for the further studies on far field behavior of flows and the detailed progress on Leray problem. To the best of our knowledge, there is no result on the far field behavior of solutions of steady Navier-Stokes system with large flux except for the axisymmetric solutions in a pipe studied in \cite{WX2}.

For viscous flows near solid boundary, besides the no-slip boundary condition,
 the Navier boundary conditions
\begin{equation}\label{BC0}
	\Bu\cdot \Bn=0,~~~\ \ \ \ (\Bn\cdot \BD(\Bu)+\alpha \Bu)\cdot \Bt=0~\text{ on }\partial\Omega,
\end{equation}
are also usually used, which were suggested by Navier \cite{Na} for the first time. Here $\BD(\Bu)$ is the strain tensor defined by
\[
 (\BD(\Bu))_{ij}=(\partial_{x_j}u_i + \partial_{x_i} u_j )/2,
\] and $\alpha\ge 0$ is the friction coefficient which measures the tendency of a fluid to slip over the boundary. $\Bt$ and $\Bn$ are the unit tangent and outer normal vector on the boundary $\partial\Omega$, respectively. If $\alpha=0$, \eqref{BC0} is also called the full slip boundary conditions. If $\alpha\to \infty$, the boundary conditions \eqref{BC0} formally reduces to the classical no-slip boundary conditions.

The Navier-Stokes system with Navier slip boundary condition has been widely studied in various aspects. One may refer to \cite{Bei1,CMR,DL,DLX,IS,Ke,LZ,MR,APS,Ga,GL,JM,WWX,XX} for some important results on nonstationary problem.
For the stationary problem, the existence and regularity of the solutions were first studied in \cite{SS}, where the Dirichlet condition and the full slip condition are imposed on different parts of the boundary of a three-dimensional interior or exterior domain. It is noteworthy that  the existence and the regularity for solutions of a generalized Stokes system with Navier boundary conditions were investigated in \cite{Bei} in some regular domain. The existence and uniqueness of very weak, weak, and strong solutions have been proved in appropriate Banach spaces in \cite{Ber}. In \cite{AR}, the existence, uniqueness, and regularity of solutions to the stationary Stokes system and also to the Navier-Stokes system with the full slip condition in both Hilbert space and $L^p$ space has been investigated. Recently,  the stationary Stokes and Navier-Stokes system with nonhomogeneous Navier boundary conditions in a bounded  three-dimensional domain  were studied in \cite{AACG}, where the existence and uniqueness for weak and strong solutions in $W^{1,p}$ and $W^{2,p}$ spaces have been established, respectively, even when the friction coefficient $\alpha$ is generalized to a function. Furthermore, the behavior of these solutions was also investigated when $\alpha$ tends to infinity (\hspace{1sp}\cite{AACG}). For more issues on the Navier slip boundary conditions, one may refer to \cite{Co,Bei2,Me}.

For flows in a nozzle with Navier-slip boundary condition, the flux across each section is also a constant, and the associated Leray problem has been studied by \cite{Mu1,Mu2,Mu3,Ko,LPY} and references therein.  In the case of three-dimensional pipes with straight outlets, a weak solution of the Navier-Stokes system with arbitrary flux has been  obtained in \cite{Ko}, which satisfies mixed boundary condition and the far field behavior \eqref{far field}. Very recently, Leray problem for flows in a pipe with Navier boundary condition was solved in \cite{LPY}, as long as the flux $\Phi$ is small and the nozzle becomes straight at large distance.

For flows in general two-dimensional channels with straight outlets, it was also proved  in \cite{Mu2} that the Navier-Stokes system has a smooth solution  with  arbitrary flux if
\[
\|\alpha-2\chi\|_{L^\infty(\partial\Omega)}\leq C(\Omega),
\] where $\chi$ is the curvature of the boundary and $C(\Omega)$ is a constant depending only on $\Omega$. However, the far behavior is not known even when the flux is small. In \cite{Mu1}, Leray problem \eqref{NS}-\eqref{BC0} with friction coefficient $\alpha=0$ was solved for any flux provided that the two-dimensional channel has straight upper boundary and coincides with the straight channel at far field. Then the exponential convergence rate of the velocity was studied in \cite{Mu3}. It's worth noting that the Dirichlet norm of the solution is finite since the corresponding shear flow $\BU$ is a constant flow in the case $\alpha=0$. The existence of solutions in a general two-dimensional channel, which may even have unbounded width, with Navier-slip boundary conditions was established in \cite{SWX2022slip} when the friction coefficient $\alpha$ is positive. When the channel tends to be flat at far fields, the uniqueness and asymptotic behavior of solutions was also established when the flux is sufficiently small (\hspace{1sp}\cite{SWX2022slip}).

In this paper, we study Leray problem with full slip boundary conditions, i.e.
\begin{equation}\label{BC}
	\Bu\cdot \Bn=0,~\ \ \ \Bn\cdot \BD(\Bu)\cdot \Bt=0~\text{ on }\partial\Omega,
\end{equation}
in a more general two-dimensional channel $\Omega$ (See Figure 1) of the form
\begin{equation}\label{defOmega}
\Omega=\{(x_1,x_2):x_1\in \mathbb{R},~f_1(x_1)<x_2<f_2(x_1)\}.
\end{equation}
Without loss of generality, assume that $f_1$ and $f_2$ are smooth functions satisfying
\begin{equation*}
f_2(t)=1 \ \ \ \text{ and }\ \ \ f_1(t)=-1  \ \ \ \text{for any }|t|\ge L,
\end{equation*}
where $L$ is a constant. The straightforward computations show that the shear flows $\BU$ in $\hat{\Omega}=\{(x_1,x_2):~x_1\in \R,~x_2\in (-1,1)\}$
with full slip boundary conditions and flux $\Phi$, i.e.,
\begin{equation*}
\int_\Sigma \BU\cdot \Bnu \,ds=\Phi,
\end{equation*}
are of the form $\BU=\frac{\Phi}{2}\Be_1$.

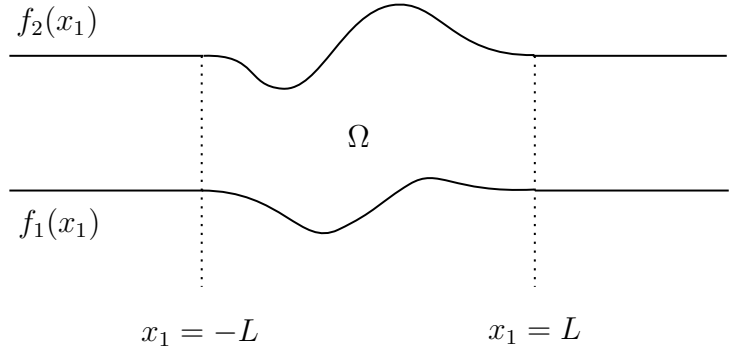
\begin{figure}[h]
\centering
\tikzset{every picture/.style={line width=0.75pt}} 

\begin{tikzpicture}[x=0.75pt,y=0.75pt,yscale=-1,xscale=1]

\draw    (100,109) -- (197.8,109) ;
\draw    (364,109) -- (461.8,109) ;
\draw    (100,177) -- (197.8,177) ;
\draw  [dash pattern={on 0.84pt off 2.51pt}]  (197,109) -- (197,187.6) -- (197,225.6) ;
\draw  [dash pattern={on 0.84pt off 2.51pt}]  (365,109) -- (365,187.6) -- (365,225.6) ;
\draw    (365,177) -- (462.8,177) ;
\draw    (198,109) .. controls (225.99,108.11) and (218.94,124.29) .. (237.8,125.6) .. controls (256.66,126.91) and (268.55,87.55) .. (292.6,83.5) .. controls (316.65,79.45) and (322.6,109.5) .. (364.8,108.6) ;
\draw    (197,177) .. controls (237.8,177) and (246.8,205.8) .. (265.8,196.8) .. controls (284.8,187.8) and (286.8,182.8) .. (301.8,173.8) .. controls (316.8,164.8) and (322.6,178.5) .. (364.6,176.5) ;

\draw (165,240.4) node [anchor=north west][inner sep=0.75pt]    {$x_{1} =-L$};
\draw (340,239.4) node [anchor=north west][inner sep=0.75pt]    {$x_{1} =L$};
\draw (269,142.4) node [anchor=north west][inner sep=0.75pt]    {$\Omega $};
\draw (103,184.4) node [anchor=north west][inner sep=0.75pt]    {$f_{1}( x_{1})$};
\draw (102,82.4) node [anchor=north west][inner sep=0.75pt]    {$f_{2}( x_{1})$};
\end{tikzpicture}

	\caption{The channel $\Omega$}
\end{figure}

We consider the solution of the form $\Bu=\Bv+\Bg$, where $\Bv\in H^1(\Omega)$ and $\Bg$ is a smooth vector field satisfying
\begin{equation}\label{flux carrier}
	\left\{\begin{aligned}
		&\operatorname{div} \Bg=0  &&\text{ in }\Omega,\\
		&\Bg\cdot \Bn=0,~ \Bn\cdot \BD(\Bg)\cdot \Bt=0  &&\text{ on }\partial\Omega,\\
		&\Bg\to \BU=\frac\Phi2 \Be_1  &&\text{ as }|x_1| \to \infty.
	\end{aligned}\right.
	\end{equation}
Using \eqref{NS}-\eqref{far field}, \eqref{BC}, and \eqref{flux carrier}, one has that $\Bv=\Bu-\Bg$ satisfies
	\begin{equation}\label{NS1}
		\left\{
		\begin{aligned}
		&-\Delta \Bv+\Bv\cdot \nabla \Bg +\Bg\cdot \nabla \Bv+\Bv\cdot \nabla \Bv  +\nabla p=\Delta \Bg-\Bg\cdot \nabla \Bg   ~~~~&&\text{ in }\Omega,\\
		&{\rm div}~\Bv=0&&\text{ in }\Omega,\\
		&\Bv\cdot \Bn=0,~\Bn\cdot \BD(\Bv)\cdot \Bt=0&&\text{ on }\partial\Omega,\\
		&\Bv\to 0 &&\text{ as }|x_1|\to \infty.
	\end{aligned}\right.
	\end{equation}

Before giving the main results of this paper, the definitions of some function spaces and the weak solution are introduced.
\begin{definition} 	Given a domain $D \subseteq \mathbb{R}^2$, denote
	\[
	L_0^2(D)=\left\{w(x): w\in L^2(D), \, \int_D w(x)dx =0 \right\}.
	\]
Given $\Omega$ defined in \eqref{defOmega}, define
\begin{equation*}
\mathcal{C}(\Omega)=\left\{\Bu\in C^\infty_c(\overline{\Omega}): ~
\Bu\cdot \Bn=0 \text{ on }\partial\Omega \right\}
\end{equation*}
and
\begin{equation*}
\mathcal{C}_\sigma(\Omega)=\left\{\Bu\in \mathcal{C}(\Omega): ~
\operatorname{div}\Bu=0 \right\}.
\end{equation*}
Let $\mcH(\Omega)$ and $\mcH_\sigma(\Omega)$  be the completions of $\mathcal{C}(\Omega)$ and  $\mathcal{C}_\sigma(\Omega)$ under $H^1$ norm, respectively.

Furthermore, for any constants $a<b$ and $0<T<\infty$, denote
\begin{equation*}
\Omega_{a,b}=\{(x_1,x_2)\in \Omega:a<x_1<b\}\  \text{ and }\ \Omega_T=\Omega_{-T,T}.
\end{equation*}
Define
\begin{equation*}
	\mathcal{C}(\Omega_{a,b})=\left\{\Bu|_{\Omega_{a,b}}:\begin{array}{l} \Bu\in C^\infty(\overline{\Omega}),~\Bu=0\text{ in }\Omega\setminus\Omega_{a,b},\\
		\Bu\cdot \Bn=0 \text{ on }\partial\Omega_{a,b}\cap \partial\Omega\end{array} \right\}
	\end{equation*}
	and
	\begin{equation*}
		\mathcal{C}_\sigma(\Omega_{a,b})=\left\{\Bu \in \mathcal{C}(\Omega_{a,b}):\operatorname{div}\Bu =0 \text{ in } \Omega_{a,b} \right\}.
		\end{equation*}
Let $\mcH(\Omega_{a,b})$ and $\mcH_\sigma(\Omega_{a,b})$  be the completions of $\mathcal{C}(\Omega_{a,b})$ and  $\mathcal{C}_\sigma(\Omega_{a,b})$  under $H^1$ norm, respectively.

Finally, denote $H_*^1(\Omega_{a,b})$ to be the set of functions in $H^1(\Omega_{a,b})$ with zero flux, i.e., for any $\Bv\in H_*^1(\Omega_{a,b})$, one has
\begin{equation}\label{zeroflux}
\int_{f_1(x_1)}^{f_2(x_1)} v_1(x_1, x_2) dx_2=0\quad \text{for any}\,\, x_1\in (a, b).
\end{equation}
\end{definition}
\begin{definition}
Assume that $\Bg$ is a smooth vector field satisfying \eqref{flux carrier}. Then a vector field $\Bu=\Bg+\Bv$ with $\Bv\in \mcH_\sigma(\Omega)$ is said to be a weak solution of the problem \eqref{NS}, \eqref{far field}, and \eqref{BC} if for any $\Bp \in \mcH_\sigma(\Omega)$, 	 $\Bv$ satisfies
\begin{equation}\label{weak solution}
\int_{\Omega}2\BD(\Bv):\BD(\Bp)+(\Bv\cdot\nabla\Bg+(\Bg+\Bv)\cdot\nabla\Bv)\cdot \Bp\,dx
=\int_{\Omega}\Delta \Bg\cdot\Bp-\Bg\cdot\nabla \Bg\cdot \Bp\,dx.
\end{equation}

\end{definition}

Then the main results of this paper can be stated as follows.
\begin{theorem}\label{bounded channel}
	Let $\Omega$ be the domain given in \eqref{defOmega}. Given any flux $\Phi\ge 0$, the Navier-Stokes system \eqref{NS}, \eqref{far field}, and \eqref{BC} has a solution $\Bu=\Bg+\Bv$, where $\Bg$ is a smooth vector field satisfying \eqref{flux carrier} and $\Bv\in \mcH_\sigma(\Omega)$ satisfies
	\begin{equation*}
		\|\Bv\|_{H^1(\Omega)}\leq C_1.
	\end{equation*}
	Furthermore, there exist positive constants $C_2$ and $C_3$ independent of $T$ such that 	for sufficiently large  $T$, one has
	\begin{equation*}
		\|\Bu-\BU\|_{H^1(\Omega\cap\{|x_1|>T\})}\leq C_3e^{-C_2^{-1}T}.
	\end{equation*}
Finally, there exists a $\Phi_0>0$ such that if the flux $\Phi\in [0,\Phi_0)$, the solution $\Bu$ is unique in the class
	\begin{equation*}
		\mathcal{S}=\{\Bw\in H_{loc}^1(\Omega):~\liminf_{t\to \infty}t^{-3}\|\nabla \Bw\|_{L^2(\Omega_t)}^2=0\}.
	\end{equation*}
\end{theorem}

There are a few remarks in order.
\begin{remark}
	The constants $C_1$, $C_2$, and $C_3$ depend only on the flux $\Phi$ and the domain $\Omega$.
\end{remark}
\begin{remark}
	Theorem \ref{bounded channel} provides a positive answer to Leray problem with full slip boundary condition and arbitrary flux.
\end{remark}
\begin{remark}
Theorem \ref{bounded channel} also holds if the channel is not flat at far field. Suppose that there exist $\gamma_1<\gamma_2,\beta$, and $L$ such that
\begin{equation}\label{1-6}
	f_1(t) =-1, \ \ f_2(t)=1 ~~\text{ for any }t\ge L
\end{equation}
and
\begin{equation}\label{1-7}
	f_1(t) =\beta t+\gamma_1, \ \ f_2(t)=\beta t+\gamma_2 ~~\text{ for any }t\leq -L.
\end{equation}
We can also construct the flux carrier $\Bg$, see Remark \ref{general case}. Then the existence, far field behavior, and uniqueness of the solutions to the problem \eqref{NS}, \eqref{far field}, and \eqref{BC} in these channels can be proved in a similar way.
\end{remark}

\begin{remark}
When the paper has been finished, we got to know that a similar result has been obtained in \cite{LPY} independently. Although there are some overlaps between the results in \cite{LPY} and that in \cite{SWX2022slip} and this paper, the analysis  is different  in many aspects.
\end{remark}


The rest of the paper is organized as follows. In Section 2, we give some important lemmas which are used here and there in the paper. Section 3 devotes to the construction of the flux carrier. In Section 4,  the existence of solutions to the problem \eqref{NS}, \eqref{far field}, and \eqref{BC} is proved by Leray-Schauder fixed point theorem. The exponential convergence rate of the $H^1$ norm of the solutions is also given in Section 4. In Section 5, we show that the solutions obtained in Section 4 is unique in $\mathcal{S}$ provided that the flux is suitably small.

\section{Preliminaries}\label{secpre}
In this section, we collect some elementary but important lemmas. We first give the Poincar\'e type inequality and embedding inequality in channels, whose proof could be found in \cite{SWX2022slip}.
\begin{lemma}\label{lemmaA1}
	For any $\Bv \in H^1_*(\Omega_{a,b})$ satisfying $\Bv \cdot \Bn =0$ on $\partial\Omega_{a,b}\cap \partial\Omega$, one has
	\begin{equation}
	\left\|\Bv\right\|_{L^2(\Omega_{a,b})}\leq M_1(\Omega_{a,b}) \left\|\nabla\Bv\right\|_{L^2(\Omega_{a,b})},
	\end{equation}
	 where
	 \begin{equation}\label{defM1}
	 M_1(\Omega_{a,b})=C\|f\|_{L^\infty(a, b)}
	  \cdot \left(1+\|f_2'\|_{L^\infty(a,b)}\right).
	 \end{equation}
	\end{lemma}
\begin{lemma}\label{lemmaA2}
	Assume that $f(x_1)=f_2(x_1)-f_1(x_1)\ge d_{a,b}>0$ for any $x_1\in (a,b)$. Then for any $\Bv \in H_*^1(\Omega_{a,b})$ satisfying  $\Bv \cdot \Bn = 0$ on $\partial\Omega_{a,b}\cap \partial\Omega$, one has
	\begin{equation*}
	\|\Bv\|_{L^4(\Omega_{a,b})}\leq M_4(\Omega_{a,b}) \| \nabla \Bv\|_{L^2(\Omega_{a,b})},
	\end{equation*}
	where
	\begin{equation}\label{defM4}
	M_4(\Omega_{a,b})=C(1+\|(f_1',f_2')\|_{L^\infty(a,b)}^2) \left(\frac{M_1}{b-a}+1\right)^{\frac12}(|\Omega_{a,b}|+(b-a)d_{a,b})^{\frac14} \left(1+\frac{M_1}{d_{a,b}}\right)
	\end{equation}
	with a universal constant $C$ and $M_1=M_1(\Omega_{a,b})$ defined in \eqref{defM1}.
\end{lemma}

Then we give the Korn inequality in the channel $\Omega$.
\begin{lemma}\label{lemmaA3}
Assume that $T>L+1$. There exists a constant $\mfc>0$ such that for any $\Bv \in \mcH_\sigma(\Omega_T)$,  it holds that
\begin{equation}\label{A3-0}
\mfc \|\nabla \Bv\|_{L^2(\Omega_T )}^2\leq 2\|\BD(\Bv)\|_{L^2(\Omega_T)}^2,
\end{equation}
where $\mfc$ is a constant independent of $T$.
\end{lemma}

\begin{proof}
Without loss of generality, we assume that $\Bv \in \mathcal{C}_\sigma(\Omega_T) $ satisfying
\be \nonumber
\int_{f_1(x_1)}^{f_2(x_1)} v_1(x_1, x_2) \, dx_2 = 0 \ \ \ \text{ for any } |x_1| < T.
\ee
 According to the formula
\begin{equation}\label{A3-1}
\Delta \Bv=2\div \BD(\Bv),
\end{equation}
integration by parts yields
\begin{equation}\label{A3-2}
\begin{aligned}
	&\int_{\Omega_T }|\nabla\Bv|^2\,dx-\int_{\partial\Omega_T \cap \partial \Omega }\Bn\cdot \nabla \Bv\cdot \Bv\,ds\\
	=&\int_{\Omega_T  }-\Delta \Bv\cdot\Bv\,dx=\int_{\Omega_T  }-2{\rm div}\BD(\Bv)\cdot\Bv\,dx\\
	=&\int_{\Omega_T  }2|\BD(\Bv)|^2\,dx -\int_{\partial\Omega_T \cap \partial  \Omega} 2\Bn\cdot\BD(\Bv)\cdot \Bv\,ds.
\end{aligned}
\end{equation}
Therefore, one has
\begin{equation*}
\int_{\Omega  }|\nabla\Bv|^2\,dx=\int_{\Omega  }2|\BD(\Bv)|^2\,dx -  \int_{\partial\Omega }2 \Bn \cdot \BD(\Bv)  \cdot \Bv  -  \Bn\cdot \nabla  \Bv\cdot \Bv\,ds.
\end{equation*}
Note that
\begin{equation*}
 \Bn\cdot \nabla \Bv\cdot \Bv=2\Bn\cdot \BD(\Bv)\cdot \Bv-\Bn\cdot\nabla\Bv\cdot\Bv.
\end{equation*}
The boundary condition $\Bv\cdot \Bn= 0$ also implies that  $\partial_{\tau }(\Bv\cdot \Bn)=0$ on the boundary $\partial\Omega$.  Hence one has
\begin{equation}\label{A3-4}
\begin{aligned}
	\Bn\cdot\nabla\Bv\cdot\Bv=&(\Bv\cdot\Bt)\partial_\Bt \Bv\cdot\Bn+(\Bv\cdot\Bn)\partial_\Bn \Bv\cdot\Bn\\
	=&(\Bv\cdot  \Bt) [\partial_\tau (\Bv \cdot \Bn)-\Bv \cdot \partial_\tau\Bn]\\
	=&-(\Bv\cdot  \Bt)(\Bv \cdot \partial_\tau\Bn) \ \ \ \ \ \ \ \ \ \ \ \ \ \ \  \mbox{on}\ \partial \Omega  .
\end{aligned}
\end{equation}
Since  $\partial_\Bt \Bn=0$ on $\partial\Omega\setminus \partial\Omega_{L+1}$, it holds that
\begin{equation}\label{A3-5}
	\begin{aligned}
\int_{\Omega_T }|\nabla\Bv|^2\,dx=& \int_{\Omega_T }2|\BD(\Bv)|^2\,dx -  \int_{\partial\Omega_T \cap \partial \Omega }(\Bv\cdot  \Bt)(\Bv \cdot \partial_\tau\Bn)\,ds\\
\leq & 2\|\BD(\Bv)\|_{L^2(\Omega_T )}^2+\int_{\partial\Omega_T \cap \partial \Omega}|\Bv|^2|\partial_\Bt\Bn|\,ds\\
\leq & 2\|\BD(\Bv)\|_{L^2(\Omega_T)}^2+C_4\|\Bv\|^2_{L^2(\partial\Omega_{L+1}\cap \partial\Omega)},
	\end{aligned}
\end{equation}
where
\begin{equation}\label{defC4}
C_4=\|\partial_\Bt\Bn\|_{L^\infty(\partial\Omega)}.
\end{equation}

Next, we claim that there exists a constant $C_5$ such that
\begin{equation}\label{A3-6}
	C_4\|\Bv\|^2_{L^2(\partial\Omega_{L+1}\cap \partial\Omega)}\leq \frac{1}{2}\|\nabla \Bv\|_{L^2(\Omega_{L+1})}^2+C_5\|\BD(\Bv)\|_{L^2(\Omega_{L+1})}^2.
\end{equation}
Otherwise, there exists a sequence $\{\Bv^m \}\subset \mcH_\sigma(\Omega_T)$ satisfying
\begin{equation*}
	C_4\|\Bv^m\|^2_{L^2(\partial\Omega_{L+1}\cap \partial\Omega)}> \frac{1}{2}\|\nabla \Bv^m\|_{L^2(\Omega_{L+1})}^2+m\|\BD(\Bv^m)\|_{L^2(\Omega_{L+1})}^2.
\end{equation*}
Define
\begin{equation*}
\Bu^m:=\frac{\Bv^m}{\|\Bv^m\|_{L^2(\partial\Omega_{L+1}\cap \partial\Omega)}}.
\end{equation*}
One has
\begin{equation*}
\|\Bu^m\|_{L^2(\partial\Omega_{L+1}\cap \partial\Omega)}=1,\ \ \|\nabla \Bu^m\|_{L^2(\Omega_{L+1})}^2<2C_4 \ \ \text{ and } \ \ \|\BD(\Bu^m)\|_{L^2(\Omega_{L+1})}^2\leq \frac{C_4}{m}.
\end{equation*}
It follows from Lemma \ref{lemmaA1} that $\{\Bu^m \}$ is also bounded in $H^1(\Omega_{L+1})$. Hence one can choose a subsequence stilled labelled by $\{\Bu^m\}$, which converges weakly in $H^1(\Omega_{L+1})$ and strongly in $L^2(\partial\Omega_{L+1}\cap \partial\Omega)$ to a vector field $\Bu^*\in H^1(\Omega_{L+1})$. Clearly, one has
\begin{equation}\label{A3-7}
\|\Bu^*\|_{L^2(\partial\Omega_{L+1}\cap \partial\Omega)}=1, \ \  \ \|\BD(\Bu^*)\|_{L^2(\Omega_{L+1} )} =0, \ \ \ \ \int_{f_1(x_1)}^{f_2(x_1)} u^*_1 \, dx_2 =0.
\end{equation}
In particular, one has
\begin{equation*}
\partial_1u_1^*=\partial_2u_2^*=0\ \ \ \text{ and }\ \ \ \partial_1u_2^*+\partial_2u_1^*=0.
 \end{equation*}
Therefore, $\Bu^*$ takes the form
\begin{equation*}
	u_1^*=ax_2+b_1,\ u_2^*=-ax_1+b_2\ \ \ \text{ for some }a\in \mathbb{R}.
\end{equation*}
On the other hand, on the  boundary $\partial \Omega_{L, L+1}\cap \partial \Omega =\{(x_1,x_2):x_1 \in (L,L+1),~x_2=\pm 1\}$,
one has $\Bu^*\cdot \Bn=u_2^*=0$ so that $a=b_1 = 0 $. This
contradicts with the first property in \eqref{A3-7}.  Finally, one combines \eqref{A3-5}-\eqref{A3-6} to conclude \eqref{A3-0} with
\begin{equation}\label{defc}
\mfc=\frac{1}{2+C_5}.
\end{equation}
 This finishes the proof of the lemma.
\end{proof}
\begin{remark}
	It is noteworthy that the constant $\mfc$ depends only on the subdomain $\Omega_{L+1}$.
\end{remark}

The following lemma on the solvability of the divergence equation is used to give the estimates involving pressure. For the proof, one may refer to \cite[Theorem  \uppercase\expandafter{\romannumeral3}.3.1 ]{Ga} and \cite{Bo}.
\begin{lemma}\label{lemmaA5}
Let $D \subset \R^n$ be a locally Lipschitz domain. Then there exists a constant $M_5$ such that for any $w\in L_0^2(D)$, the problem
\begin{equation}\label{A5-1}
\left\{\begin{aligned}
{\rm div}~\Ba=w ~~~~~~~~~~&\text{ in }D,\\
\Ba=0 ~~~~~~~~~~~~~~~~~&\text{ on }\partial D
\end{aligned}\right.
\end{equation}
has a solution $\Ba \in H^1_0(D)$ satisfying
\[
\|\nabla\Ba\|_{L^2(D)}\leq M_5(D)\|w\|_{L^2(D)}.
\]  In particular, if the domain $D$ is star-like with respect to some open ball $B$ with  $\overline{B}\subset D$, then the constant $M_5(D)$ admits the following estimate
\begin{equation*}
M_5(D)\leq C \left(\frac{R_0}{R}\right)^n\left(1+\frac{R_0}{R}\right),
\end{equation*}
where $R_0$ is the diameter of the domain $D$ and $R$ is the radius of the ball $B$.
\end{lemma}
\begin{remark}
In particular, for $D=\Omega_{t-1,t}$ or $\Omega_{-t, -t+1}$, $t>L+1$, the constant $M_5(D)$ is independent of $t$ since $D$ is a star-like domain with respect to a ball with radius $\frac14$.
\end{remark}

We next recall a differential inequality (cf.\cite{LS}), which plays the key role in establishing the uniqueness  of the solutions.
\begin{lemma}\label{lemmaA4}~
	 Let $z(t)$ be a nondecreasing and nonnegative function, which is not identically equal to zero. Assume that $\Psi(\tau)$ is a monotonically increasing function, which equals to zero at $\tau =0$ and tends to $\infty$ as $\tau \rightarrow \infty$. Suppose that there exist $m>1, t_0\ge 0, \tau_1\ge 0,c_0>0$ such that
	\begin{equation*}
	  z(t)\leq \Psi(z'(t))~~~~\text{ for any } t\ge t_0\text{ and }
	\Psi(\tau)\leq c_0\tau^m \text{ for any }\tau\ge \tau_1,
	\end{equation*}
	then it holds that
	\begin{equation*}
	  \liminf_{t\to \infty} t^\frac{-m}{m-1}z(t)>0.
	\end{equation*}
\end{lemma}
With the aid of the differential inequality for the Dirichlet norm on approximate domain $\Omega_t$, one has that either it is trivial or it grows faster that $t^{\frac{m}{m-1}}$.

\section{Flux carrier}\label{secflux}
In this section, we construct the so called flux carrier $\Bg=(g_1,g_2)$, which is a smooth vector field satisfying
\begin{equation}\label{2-0}
\left\{\begin{aligned}
	&\operatorname{div} \Bg=0  &&\text{ in }\Omega,\\
	&\Bg\cdot \Bn=0,~ \Bn\cdot \BD(\Bg)\cdot \tau=0  &&\text{ on }\partial\Omega,\\
	&\Bg\to \BU=\frac\Phi2  \Be_1  &&\text{ as }|x_1| \to \infty.
\end{aligned}\right.
\end{equation}

Inspired by \cite{A1,Mu1}, we introduce two smooth functions $\mu(t;\varepsilon):[0,\infty)\to [0,1]$ and $\pi(s;\mfD):\R \to [0,1]$ satisfying
\begin{equation}\label{2-1}
\mu(t;\varepsilon)=\left\{
\begin{aligned}
&1,\,\,\,\,\,\, \text{ if }t\text{ near }0,\\
&0,\,\,\,\,\,\, \text{ if }t\ge \varepsilon,
\end{aligned}
\right.
\ \ \ \ \ \
	\pi(s;\mfD)=\left\{
	\begin{aligned}
	&0,\,\,\,\,\,\, \text{ if }|s|\leq \frac{5\mfD}{4},\\
	&1,\,\,\,\,\,\, \text{ if }|s|\ge \frac{7\mfD}{4}
	\end{aligned}
	\right.
	\end{equation}
and
\begin{equation}\label{2-2-1}
	0\leq -\mu'(t;\varepsilon)\leq \frac{\varepsilon}{t}, \ \ \ \ 	0 \leq \pi'(s;\mfD)\leq \frac{4}{\mfD},\ \ \ \ 	0 \leq \pi''(s;\mfD)\leq \frac{16}{\mfD^2},
\end{equation}
where $\varepsilon$ and $\mfD>L$ are two parameters to be determined. One can refer to \cite[Lemma 2.6]{A1} for the detailed construction of $\mu(t;\varepsilon)$. Define $\Bg=(g_1,g_2)$ as
\begin{equation}\label{2-12}
	g_1(x_1,x_2)=\partial_{x_2}G(x_1,x_2;\varepsilon)+\left(\frac\Phi2-\partial_{x_2}G(x_1,x_2;\varepsilon)\right)\pi(x_1;\mfD)
\end{equation}
and
\begin{equation}\label{2-11}
g_2(x_1,x_2)=\left\{\begin{aligned}&-\partial_{x_1}G(x_1,x_2;\varepsilon)  &&\text{ if }|x_1|< \mfD,\\
	&\pi'(x_1;\mfD)\left(G(x_1,x_2;\varepsilon)-\frac\Phi2 (x_2+1)\right) &&\text{ if } |x_1|\ge \mfD,
\end{aligned}\right.
\end{equation}
where
\begin{equation}\label{defG}
	G(x_1,x_2;\varepsilon)=\Phi \mu(f_2(x_1)-x_2;\varepsilon).
\end{equation}

Denote
\begin{equation}\label{defSigma}
	\Sigma(x_1)=\{(x_1,x_2): f_1(x_1)<x_2<f_2(x_1)\}.
\end{equation}
In order to show that $\Bg \in C^\infty(\overline{\Omega})$, it's sufficient to verify the smoothness of $g_2$ near $\Sigma(\pm \mfD)$ since both $\pi$ and $\mu$ are smooth.  Actually, it holds that $f_2(x_1)=1$ for any $|x_1|>L$ and then the function
\begin{equation*}
	G(x_1,x_2;\varepsilon)=\mu(f_2(x_1)-x_2;\varepsilon,\mfD)=\Phi\mu(1-x_2;\varepsilon)
\end{equation*}
depends only on $x_2$ in the subdomain $\Omega\setminus\Omega_{L}$. Therefore, for any $\Bx \in \Omega$ with $L\leq |x_1|<\mfD$, one has
\begin{equation*}
	g_2(x_1,x_2)=-\partial_{x_1}G(x_1,x_2;\varepsilon)=-\partial_{x_1}\mu(1-x_2;\varepsilon,\mfD)=0.
\end{equation*}
On the other hand, \eqref{2-1}, together with \eqref{2-11}, implies that
$g_2(x_1,x_2)\equiv 0$ for any $\Bx \in \Omega$ with $\mfD\leq |x_1|<\frac{5\mfD}{4}$. Hence $\Bg\in C^\infty(\overline{\Omega})$.

Next, note that
\begin{equation*}
G(x_1,x_2;\varepsilon)=\left\{
\begin{aligned}
&\Phi,&&\text{ if }x_2\text{ near }f_2(x_1),\\
&0,&&\text{ if }x_2\leq f_2(x_1)-\varepsilon,
\end{aligned}
\right.
\end{equation*}
and
\begin{equation*}
	\Bg=(\partial_{x_2}G,-\partial_{x_1}G) \text{ in }\Omega_\mfD.
\end{equation*}
Hence $\Bg$ is a solenoidal vector field with flux $\Phi$ in $\Omega_\mfD$. In particular, $\Bg$ vanishes near the boundary $\partial \Omega\cap \partial\Omega_{\mfD}$.

In the subdomain $\Omega\setminus\Omega_\mfD$, since $f_2(x_1)=1$ and $f_1(x_1)=-1$ for any $|x_1|\ge L$, one has $\partial_{x_1} G=0$. It follows from straightforward computations that one has
\begin{equation*}
\begin{aligned}
		\operatorname{div}\Bg = &\partial_{x_1}g_1+\partial_{x_2}g_2\\
		=&~\left(\frac\Phi2-\partial_{x_2}G(x_1,x_2;\varepsilon)\right)\pi'(x_1;\mfD)+\pi'(x_1;\mfD)\left(\partial_{x_2}G(x_1,x_2;\varepsilon)-\frac\Phi2 \right)\\
		=&~0
\end{aligned}
\end{equation*}
and
\begin{equation*}
\begin{aligned}
\int_{\Sigma(x_1)} g_1(x_1,x_2)\,dx_2= &~\int_{-1}^1 \partial_{x_2}G(x_1,x_2;\varepsilon)+\left(\frac{\Phi}{2}-\partial_{x_2}G(x_1,x_2;\varepsilon)\right)\pi(x_1;\mfD)\,dx_2\\
=&~ \Phi+\pi(x_1;\mfD)\int_{-1}^1\frac{\Phi}{2}-\partial_{x_2}G(x_1,x_2;\varepsilon)\,dx_2\\
=&~\Phi.
\end{aligned}
\end{equation*}
Moreover, at the upper boundary
\begin{equation*}
	S_{2;\mfD} = \{\Bx \in \partial\Omega: x_2=1, ~|x_1|>\mfD\},
\end{equation*}
one has $\Bt=(1,0)$, $\Bn=(0,1)$. Note also that $G(x_1,x_2;\varepsilon)|_{x_2=f_2(x_1)}=\Phi$ and  $\partial_{x_2} G(x_1,x_2;\varepsilon)$ vanishes near the boundary $\partial\Omega$. Hence it holds that
\begin{equation*}
	\Bg \cdot \Bn=g_2 (x_1,x_2)|_{x_2=f_2(x_1)}=\pi'(x_1;\mfD) (G(x_1,1;\varepsilon)-\Phi)=0
	\end{equation*}
	and
\begin{equation*}
\begin{aligned}
\Bn\cdot \BD(\Bg )\cdot \Bt=&\frac12(\partial_{x_2}g_1 +\partial_{x_1}g_2 )(x_1,x_2)|_{x_2=f_2(x_1)}\\
=&~\frac12 \left(\partial_{x_2}^2G(x_1,1;\varepsilon)-\partial_{x_2}^2G(x_1,1;\varepsilon)\pi(x_1;\mfD)+\pi''(x_1;\mfD) (G(x_1,1;\varepsilon)-\Phi)\right)\\
=&~0.
\end{aligned}
\end{equation*}
Similarly, at the lower boundary
\begin{equation*}
	S_{1;\mfD} = \{\Bx \in \partial\Omega: x_2=-1, ~ |x_1|>\mfD\},
\end{equation*}
one has also $\Bt=(1,0)$, $\Bn=(0,-1)$ and $G(x_1,x_2;\varepsilon)|_{x_2=f_1(x_1)}=0$. Therefore, one has
\begin{equation*}
\Bg \cdot \Bn=-g_2 (x_1,x_2;\varepsilon,\mfD)|_{x_2=f_1(x_1)}=\pi'(x_1;\mfD) G(x_1,-1;\varepsilon)=0
\end{equation*}
and
\begin{equation*}
\begin{aligned}
\Bn\cdot \BD(\Bg )\cdot \Bt=&-\frac12(\partial_{x_2}g_1 +\partial_{x_1}g_2 )(x_1,x_2)|_{x_2=f_1(x_1)}\\
=&~-\frac12 \left(\partial_{x_2}^2G(x_1,-1;\varepsilon)-\partial_{x_2}^2G(x_1,-1;\varepsilon)\pi(x_1;\mfD)+\pi''(x_1;\mfD)G(x_1,-1;\varepsilon)\right)\\
=&~0.
\end{aligned}
\end{equation*}
Finally, noting $\pi(x_1;\mfD)= 1$ for any $|x_1|\ge \frac{7\mfD}{4}$, one has
\begin{equation*}
	\Bg\equiv \left(\frac\Phi2,0\right)  \ \ \ \text{ in }\Omega\setminus \Omega_{\frac{7\mfD}{4}}.
\end{equation*}
Hence $\Bg$ satisfies \eqref{2-0} in $\Omega$.

\begin{remark}\label{general case}
	For a more general channel domain $\Omega$ with \eqref{1-6}-\eqref{1-7}, one could also construct the corresponding flux carrier $\Bg$ via some modifications. Assume $\beta>0$. For any  $\Bx\in \Omega_{-L,\infty}$, we define $\Bg=(g_1,g_2)$ as the same form of \eqref{2-12} and \eqref{2-11}, i.e.,
	\begin{equation*}
		g_1(x_1,x_2)=\partial_{x_2}G(x_1,x_2;\varepsilon)+\left(\frac\Phi2-\partial_{x_2}G(x_1,x_2;\varepsilon)\right)\pi(x_1;\mfD) \text{ if }x_1\ge -L
	\end{equation*}
	and
	\begin{equation*}
		g_2(x_1,x_2)=\left\{\begin{aligned}&-\partial_{x_1}G(x_1,x_2;\varepsilon),  &&\text{ if }-L\leq x_1< \mfD,\\
			&\pi'(x_1;\mfD)\left(G(x_1,x_2;\varepsilon)-\frac\Phi2 (x_2+1)\right), &&\text{ if } x_1\ge \mfD.
		\end{aligned}\right.
	\end{equation*}
	
	On the other hand, let
	\begin{equation}\label{definerotation}
	\tilde{x}_1 =x_1 \cos\theta +x_2\sin\theta\quad \text{and}\quad 	\tilde{x}_2 =-x_1 \sin\theta +x_2\cos\theta,
	\end{equation}
	with  $\theta=\arctan \beta$, which transforms the outlet $\Omega_{-\infty,-L}$ into a flat outlet $\tilde{\Omega}_{-\infty,-L}$ in the new coordinate $(\tilde{x}_1, \tilde{x}_2)$. More precisely,
	\begin{equation*}
		\tilde{\Omega}_{-\infty,-L}=\left\{(\tilde{x}_1,\tilde{x}_2):~\tilde{x}_2\in (\gamma_1\cos\theta,\gamma_2\cos\theta),~\tilde{x}_1<\tilde{x}_2 \tan\theta -\frac{L}{\cos\theta}\right\}.
	\end{equation*}
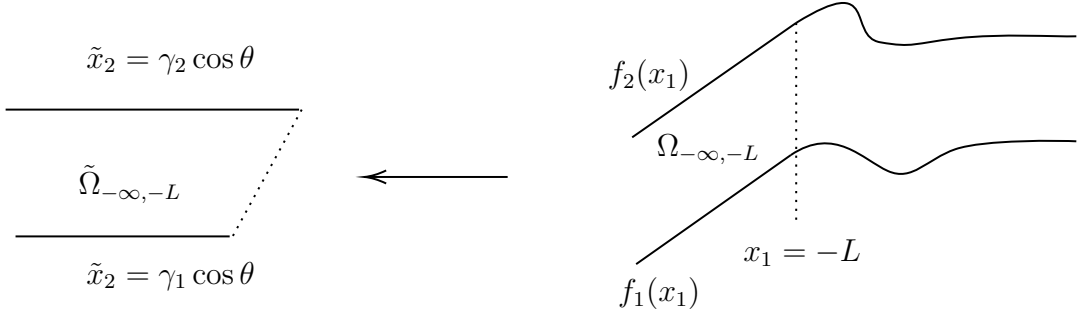
\begin{figure}[h]
	\centering

	\tikzset{every picture/.style={line width=0.75pt}} 
	
	\begin{tikzpicture}[x=0.75pt,y=0.75pt,yscale=-1,xscale=1]
		
		\draw    (483.8,118.95) .. controls (505.99,103.98) and (520.32,121.45) .. (532.8,127) .. controls (545.28,132.55) and (552.24,120.99) .. (566.8,116) .. controls (581.36,111.01) and (619.91,111.62) .. (627.8,112) ;
		\draw    (403.8,110) -- (421.1,97.92) -- (424.56,95.5) -- (486.52,52.22) ;
		\draw    (486.52,52.22) .. controls (528.8,26) and (511.8,59) .. (528.8,62) .. controls (545.8,65) and (545.49,62.5) .. (567.8,60) .. controls (590.11,57.5) and (616.47,59.31) .. (627.8,59) ;
		\draw  [dash pattern={on 0.84pt off 2.51pt}]  (486.52,52.22) -- (486.52,88.61) -- (486.52,153.91) ;
		\draw    (87.8,96) -- (235.8,96) ;
		\draw    (92.8,160) -- (200.8,160) ;
		\draw  [dash pattern={on 0.84pt off 2.51pt}]  (237.14,96) -- (224.67,118.9) -- (202.29,160) ;
		\draw    (341,130) -- (270.8,130) ;
		\draw [shift={(268.8,130)}, rotate = 360] [color={rgb, 255:red, 0; green, 0; blue, 0 }  ][line width=0.75]    (10.93,-3.29) .. controls (6.95,-1.4) and (3.31,-0.3) .. (0,0) .. controls (3.31,0.3) and (6.95,1.4) .. (10.93,3.29)   ;
		\draw    (405.8,174) -- (422.11,162.49) -- (425.38,160.18) -- (483.8,118.95) ;
		
		\draw (390.09,69.19) node [anchor=north west][inner sep=0.75pt]    {$f_{2}( x_{1})$};
		\draw (395.26,179.11) node [anchor=north west][inner sep=0.75pt]    {$f_{1}( x_{1})$};
		\draw (458.37,160.01) node [anchor=north west][inner sep=0.75pt]    {$x_{1} =-L$};
		\draw (291,104.4) node [anchor=north west][inner sep=0.75pt]    { };
		\draw (123,122.4) node [anchor=north west][inner sep=0.75pt]    {$\tilde{\Omega} _{-\infty ,-L}$};
		\draw (414.8,106.81) node [anchor=north west][inner sep=0.75pt]    {$\Omega _{-\infty ,-L}$};
		\draw (127.09,62.19) node [anchor=north west][inner sep=0.75pt]    {$\tilde{x}_2=\gamma_2\cos\theta$};
		\draw (127.09,171.19) node [anchor=north west][inner sep=0.75pt]    {$\tilde{x}_2=\gamma_1\cos\theta$};

	\end{tikzpicture}

	\caption{Rotation transformation}
\end{figure}

 In the flat outlet $\tilde{\Omega}_{-\infty,-L}$, one could  construct the vector field $\tilde{\Bg}=(\tilde{g}_1,\tilde{g}_2)$  in a way similar to \eqref{2-12}-\eqref{2-11},
	\begin{equation*}
		\tilde{g}_1(\tilde{x}_1,\tilde{x}_2)=\partial_{\tilde{x}_2} \tilde{G}(\tilde{x}_1, \tilde{x}_2)+\left(\frac{\Phi}{(\gamma_2-\gamma_1)\cos\theta}-\partial_{x_2}\tilde{G}(\tilde{x}_1, \tilde{x_2})\right)\pi(\tilde{x}_1-\tilde{L}_1-\mfD;\mfD)\ \ \  \text{ if }\tilde{x}_1<\tilde{L}_2
	\end{equation*}
	and
	\begin{equation*}
		\tilde{g}_2(\tilde{x}_1,\tilde{x}_2)=\left\{\begin{aligned}&-\partial_{\tilde{x}_1} \tilde{G}(\tilde{x}_1, \tilde{x}_2)   &&\text{ if }\tilde{L}_1\leq \tilde{x}_1< \tilde{L}_2,\\
			&\pi'(\tilde{x}_1-\tilde{L}_1-\mfD;\mfD)\left(\tilde{G}(\tilde{x}_1, \tilde{x}_2)-\frac{\Phi(\tilde{x}_2-\gamma_1\cos\theta)}{(\gamma_2-\gamma_1)\cos\theta} \right)  &&\text{ if } \tilde{x}_1<\tilde{L}_1,
		\end{aligned}\right.
	\end{equation*}
	where
	\begin{equation*}
		\tilde{L}_2 =\gamma_2\sin\theta-\frac{L}{\cos\theta},\ \ \ \ \tilde{L}_1=\gamma_1\sin\theta-\frac{L}{\cos\theta},
	\end{equation*}
	and
	\begin{equation*}
		\tilde{G}(\tilde{x}_1,\tilde{x}_2)=\Phi \mu\left(\frac{\gamma_2 \cos\theta-\tilde{x}_2}{\cos\theta};\varepsilon\right).
	\end{equation*}
For $x_1<-L$, define
	\begin{equation*}
	g_1(x_1,x_2)=\tilde{g}_1(\tilde{x}_1, \tilde{x}_2)\cos \theta -\tilde{g}_2(\tilde{x}_1, \tilde{x}_2)\sin\theta,\quad 	g_2(x_1,x_2)=\tilde{g}_1(\tilde{x}_1, \tilde{x}_2)\sin \theta +\tilde{g}_2(\tilde{x}_1, \tilde{x}_2)\cos\theta,
	\end{equation*}
where the relation between $(x_1, x_2)$ and $(\tilde{x}_1, \tilde{x}_2)$ is given in \eqref{definerotation}.
Then the straightforward computations show that $\Bg=(g_1,g_2)$ is smooth near $\Sigma(-L)$ and thus is smooth in $\Omega$. Furthermore, $\Bg$ is divergence free, satisfies the slip boundary condition on the channel boundary, and tends to the associated shear flows with flux $\Phi$ at far fields. Hence $\Bg$ is a flux carrier.
\end{remark}

The following two  lemmas give the crucial properties of the flux carrier $\Bg$, which plays an important role in the energy estimates.
\begin{lemma}\label{lemma1}
The function $G(x_1,x_2;\varepsilon)$ defined  in \eqref{defG} satisfies
\begin{equation*}
|\nabla G(x_1,x_2;\varepsilon)|+|\nabla^2G(x_1,x_2;\varepsilon)|\leq C(\varepsilon)\Phi.
\end{equation*}
Furthermore, for any function $w\in H^1(\Omega_{a,b})$ satisfying $w=0$ on the upper boundary $S_{2;a,b}:=\{\Bx\in \partial \Omega:\ x_2=f_2(x_1), \ a<x_1<b\}$, it holds that
\begin{equation*}
\int_{\Omega_{a,b}} w^2 |\partial_{x_2}G|^2\,dx\leq C \Phi^2\varepsilon^2\int_{\Omega_{a,b}}|\partial_{x_2} w |^2\,dx,
\end{equation*}
where $C(\varepsilon)$ is a constant depending only on $\varepsilon$ and $C$ is a uniform constant independent of $\epsilon$.
\end{lemma}
\begin{proof}
Recall the definition for $G(x_1,x_2;\varepsilon)$ in \eqref{defG}. It follows from direct computations that one has
\begin{equation}\label{2-3}
	\partial_{x_1}G(x_1,x_2;\varepsilon)=\Phi \mu'(f_2(x_1)-x_2;\varepsilon)f_2'(x_1),\ \
	\partial_{x_2}G(x_1,x_2;\varepsilon)=-\Phi \mu'(f_2(x_1)-x_2;\varepsilon).
\end{equation}
Furthermore,
\begin{equation}\label{2-5}
	\partial_{x_1x_2}^2G(x_1,x_2;\varepsilon)=-\Phi \mu''(f_2(x_1)-x_2;\varepsilon)f_2'(x_1),
\end{equation}
\begin{equation}\label{2-6}
	\partial_{x_2}^2G(x_1,x_2;\varepsilon)=\Phi \mu''(f_2(x_1)-x_2;\varepsilon)
\end{equation}
and
\begin{equation}\label{2-7}
	\partial_{x_1}^2G(x_1,x_2;\varepsilon)=\Phi \mu'(f_2(x_1)-x_2;\varepsilon)f_2''(x_1)+\Phi \mu''(f_2(x_1)-x_2;\varepsilon)|f_2'(x_1)|^2.
\end{equation}
Noting $\mu(t;\varepsilon)$ is smooth and $ \operatorname{supp} \mu'\subset [0,\varepsilon]$, one has
\begin{equation*}
	|\mu'(t;\varepsilon)|,~|\mu''(t;\varepsilon)|\leq C(\varepsilon).
\end{equation*}
Moreover, since $f_2(x_1)=1$ for any $|x_1|\ge L$, one has also
\begin{equation*}
	|f_2'(x_1)| ,~|f_2''(x_1)|\leq C.
\end{equation*}
Then it follows that
\begin{equation*}
	|\nabla G(x_1,x_2;\varepsilon)|+|\nabla^2G(x_1,x_2;\varepsilon)|\leq C(\varepsilon)\Phi.
\end{equation*}

Next, one has
\begin{equation*}\begin{aligned}
\int_{\Omega_{a,b}} w^2 |\partial_{x_2}G|^2\,dx=&\int_{\Omega_{a,b}} \Phi^2(\mu'(f_2(x_1)-x_2;\varepsilon))^2w^2\,dx\\
	\leq &C\Phi^2\varepsilon^2\int_a^b\,dx_1\int_{f_1(x_1)}^{f_2(x_1)}\frac{w^2}{(f_2(x_1)-x_2)^2}\,dx_2\\
\leq&C\Phi^2\varepsilon^2\int_{\Omega_{a,b}}|\partial_{x_2}w|^2\,dx,
\end{aligned}\end{equation*}
where \eqref{2-2-1} and the Hardy inequality (\cite{HLP}) have been used. This finishes the proof of the lemma.
\end{proof}

\begin{lemma}\label{lemma2}
The flux carrier $\Bg$ satisfies
\begin{equation}\label{2-8}
	\int_{\Omega}|\nabla  \Bg|^2+|\Bg\cdot \nabla \Bg|^2\,dx\leq C(\varepsilon,\mfD)(\Phi^2 + \Phi^4),
\end{equation}
where $C(\varepsilon,\mfD)$ is a constant depending only on $\varepsilon$ and $\mfD$. Moreover, for any $\delta>0$, there exist $\varepsilon$ and $\mfD$ such that for any $\Bv \in \mcH_\sigma(\Omega)$, it holds that
	\begin{equation*}
	\left|\int_{\Omega} \Bv\cdot \nabla \Bg\cdot \Bv\,dx\right|\leq \delta \|\nabla \Bv\|_{L^2}^2.
	\end{equation*}
\end{lemma}
\begin{proof}
	Noting $\Bg=\frac\Phi2 \Be_1$ for any $\Bx\in \Omega$ with $|x_1|\ge 2\mfD$, one has
	\begin{equation}\label{2-9}
		\int_{\Omega}|\nabla \Bg|^2+|\Bg\cdot \nabla \Bg|^2\,dx= \int_{\Omega_{2\mfD}}|\nabla \Bg|^2+|\Bg\cdot \nabla \Bg|^2\,dx\leq |\Omega_{2\mfD}| \sup_{\Bx\in \Omega_{2\mfD}} \left( |\nabla \Bg|^2+|\Bg\cdot \nabla \Bg|^2 \right).
	\end{equation}
Using \eqref{2-12}-\eqref{2-11} and Lemma \ref{lemma1}, one has
\begin{equation*}
	\sup_{\Bx\in \Omega_{2\mfD}} \left( |\nabla \Bg|^2+|\Bg\cdot \nabla \Bg|^2 \right) \leq C(\varepsilon,\mfD)(\Phi^2 + \Phi^4).
\end{equation*}
This, together with \eqref{2-9}, gives \eqref{2-8}. Next, from \eqref{2-3}, one has the following equality
\begin{equation}\label{2-12-1}
	\partial_{x_1}G(x_1,x_2;\varepsilon)=-f_2'(x_1)\partial_{x_2}G(x_1,x_2;\varepsilon).
\end{equation}
Using \eqref{2-12} gives
\begin{equation}\label{2-15}
\begin{aligned}
&\int_{\Omega}\Bv \cdot\nabla g_1 v_1\,dx\\
=&\int_{\Omega}(v_1\partial_{x_1}+v_2\partial_{x_2})\left(\partial_{x_2}G(x_1,x_2;\varepsilon)+\left(\frac\Phi2-\partial_{x_2}G(x_1,x_2;\varepsilon)\right)\pi(x_1;\mfD)\right)v_1\,dx\\
=&\int_{\Omega}(v_1\partial_{x_1}+v_2\partial_{x_2})\left((1-\pi(x_1;\mfD))\partial_{x_2}G(x_1,x_2;\varepsilon)+\frac\Phi2\pi(x_1;\mfD)\right)v_1\,dx\\
=&\int_{\Omega}(v_1^2\partial_{x_1}+v_1v_2\partial_{x_2})\left((1-\pi(x_1;\mfD))\partial_{x_2}G(x_1,x_2;\varepsilon)\right)\,dx+\int_{\Omega}\frac\Phi2 v_1^2\pi'(x_1;\mfD)\,dx.
\end{aligned}
\end{equation}
It follows from \eqref{2-1} and Lemma \ref{lemmaA1} that
\begin{equation}\label{2-21}
\left|\int_{\Omega}\frac\Phi2 v_1^2\pi'(x_1;\mfD)\,dx\right|\leq \frac{C\Phi}{\mfD} \|\Bv\|_{L^2(\Omega)}^2\leq \frac{C\Phi}{\mfD} \|\nabla\Bv\|_{L^2(\Omega)}^2.
\end{equation}
Noting $\partial_{x_2}G(x_1,x_2;\varepsilon)$ vanishes near the boundary $\partial\Omega$ and $\Bv$ is divergence free in $\Omega$, one uses integration by parts to obtain
\begin{equation}\label{2-21-1}
\begin{aligned}
&\int_{\Omega}(v_1^2\partial_{x_1}+v_1v_2\partial_{x_2})\left((1-\pi(x_1;\mfD))\partial_{x_2}G(x_1,x_2;\varepsilon)\right)\,dx\\
=&-\int_{\Omega}(v_1\partial_{x_1} v_1+v_2\partial_{x_2}v_1)(1-\pi(x_1;\mfD))\partial_{x_2}G(x_1,x_2;\varepsilon)\,dx\\
=&-\int_{\Omega}(1-\pi(x_1;\mfD))\left(v_1\partial_{x_1} v_1\partial_{x_2}G(x_1,x_2;\varepsilon) +v_2\partial_{x_2}v_1\partial_{x_2}G(x_1,x_2;\varepsilon)\right)\,dx\\
=&-\int_{\Omega}(1-\pi(x_1;\mfD))\left(v_1\partial_{x_1} v_1\partial_{x_2}G(x_1,x_2;\varepsilon)-v_1\partial_{x_2} v_1\partial_{x_1}G(x_1,x_2;\varepsilon)\right)\,dx\\
&-\int_{\Omega}(1-\pi(x_1;\mfD))\left(v_2\partial_{x_2}v_1\partial_{x_2}G(x_1,x_2;\varepsilon)+v_1\partial_{x_2} v_1\partial_{x_1}G(x_1,x_2;\varepsilon) \right)\,dx\\
=&-\int_{\Omega}\frac12 (1-\pi(x_1;\mfD))\left(\partial_{x_1} (v_1^2)\partial_{x_2}G(x_1,x_2;\varepsilon)-\partial_{x_2} (v_1^2)\partial_{x_1}G(x_1,x_2;\varepsilon)\right)\,dx\\
&-\int_{\Omega}(1-\pi(x_1;\mfD))\left(v_2\partial_{x_2}v_1\partial_{x_2}G(x_1,x_2;\varepsilon)+v_1\partial_{x_2} v_1\partial_{x_1}G(x_1,x_2;\varepsilon) \right)\,dx\\
=&-\int_{\Omega}(1-\pi(x_1;\mfD))\partial_{x_2}v_1 \left(v_2\partial_{x_2}G(x_1,x_2;\varepsilon)+v_1\partial_{x_1}G(x_1,x_2;\varepsilon) \right)\,dx\\
&-\int_{\Omega}\frac12 \pi'(x_1;\mfD)v_1^2 \partial_{x_2}G(x_1,x_2;\varepsilon)\,dx\\
=&-\int_{\Omega}(1-\pi(x_1;\mfD))\partial_{x_2}v_1 (v_2-v_1 f_2'(x_1))\partial_{x_2}G(x_1,x_2;\varepsilon)\,dx\\
&-\int_{\Omega}\frac12 \pi'(x_1;\mfD)v_1^2 \partial_{x_2}G(x_1,x_2;\varepsilon)\,dx,
\end{aligned}
\end{equation}
where the equality \eqref{2-12-1} has been used to get the last equality. Note that on the upper boundary $S_2=\{\Bx\in \partial\Omega:~x_1\in \R, ~x_2=f_2(x_1)\}$, the impermeability condition $\Bv\cdot\Bn=0$ can be written as
\begin{equation*}
v_2(x_1,f_2(x_1))-f_2'(x_1)v_1(x_1,f_2(x_1))=0.
\end{equation*}
Then applying Cauchy-Schwarz inequality and Lemma \ref{lemma1} gives
\begin{equation}\label{2-17}
\begin{aligned}
&\left|\int_{\Omega}(1-\pi(x_1;\mfD))\partial_{x_2}v_1 (v_2-v_1 f_2'(x_1))\partial_{x_2}G(x_1,x_2;\varepsilon)\,dx\right|\\
\leq&\|\partial_{x_2}v_1\|_{L^2(\Omega)}\left(\int_{\Omega}|(v_2-v_1 f_2'(x_1))\partial_{x_2}G(x_1,x_2;\varepsilon)|^2\,dx\right)^\frac12\\
\leq& C\varepsilon\Phi\|\partial_{x_2}v_1\|_{L^2(\Omega)}\|\partial_{x_2}(v_2-v_1 f_2'(x_1))\|_{L^2(\Omega)}\\
\leq& C\varepsilon\Phi\|\nabla\Bv\|_{L^2(\Omega)}^2.
\end{aligned}
\end{equation}
Lemmas \ref{lemmaA1} and  \ref{lemma1}, together with \eqref{2-2-1}, yield
\begin{equation}\label{2-17-1}
\left|\int_{\Omega}\frac12 \pi'(x_1;\mfD)v_1^2 \partial_{x_2}G(x_1,x_2;\varepsilon)\,dx\right|
\leq \frac{C(\varepsilon)\Phi}{\mfD}\|\Bv\|_{L^2(\Omega)}^2 \leq \frac{C(\varepsilon)\Phi}{\mfD}\|\nabla \Bv\|_{L^2(\Omega)}^2 .
\end{equation}

On the other hand, with the aid of the explicit form in \eqref{2-11}, one has
\begin{equation}\label{2-20}
\begin{aligned}
&\int_{\Omega}\Bv \cdot\nabla g_2 v_2\,dx=\int_{\Omega_{\mfD}}-(v_1v_2\partial_{x_1}+v_2^2\partial_{x_2})\partial_{x_1}G(x_1,x_2;\varepsilon)\,dx\\
&+\int_{\Omega\setminus\Omega_{\mfD}}(v_1v_2\partial_{x_1}+v_2^2\partial_{x_2})\left[\pi'(x_1;\mfD) \left(G(x_1,x_2;\varepsilon)-\frac\Phi2(x_2+1)\right)\right]\,dx.
\end{aligned}
\end{equation}
Since $\partial_{x_1}G(x_1,x_2;\varepsilon)=\Phi\partial_{x_1}\mu(1-x_2;\varepsilon)=0$ near  $\Sigma(\pm \mfD)$ and $\partial_{x_1}G(x_1,x_2;\varepsilon)$ vanishes near the boundary $\partial\Omega\cap \partial\Omega_{\mfD}$, the integration by parts together with \eqref{2-12-1} gives
\begin{equation*}
\begin{aligned}
	&\int_{\Omega_{\mfD}}-(v_2v_1\partial_{x_1}+v_2^2\partial_{x_2})\partial_{x_1}G(x_1,x_2;\varepsilon)\,dx\\
	=&\int_{\Omega_{\mfD}}(v_1\partial_{x_1}v_2+v_2\partial_{x_2}v_2)\partial_{x_1}G(x_1,x_2;\varepsilon)\,dx\\
	=&\int_{\Omega_{\mfD}}v_1\partial_{x_1}v_2\partial_{x_1}G(x_1,x_2;\varepsilon)+ v_2\partial_{x_1}v_2\partial_{x_2}G(x_1,x_2;\varepsilon)\,dx\\
	&+\int_{\Omega_{\mfD}}\partial_{x_2}\left(\frac{v_2^2}{2}\right)\partial_{x_1}G(x_1,x_2;\varepsilon)-\partial_{x_1}\left(\frac{v_2^2}{2}\right)\partial_{x_2}G(x_1,x_2;\varepsilon)\,dx\\
	=&\int_{\Omega_{\mfD}}\partial_{x_1}v_2 \left[ v_1\partial_{x_1}G(x_1,x_2;\varepsilon)+ v_2\partial_{x_2}G(x_1,x_2;\varepsilon)\right]\,dx\\
	=&\int_{\Omega_{\mfD}}\partial_{x_1}v_2(v_2-f_2'(x_1)v_1)\partial_{x_2}G(x_1,x_2;\varepsilon)\,dx.
\end{aligned}
\end{equation*}
Therefore, similar to \eqref{2-17}, one uses Cauchy-Schwarz inequality and Lemma \ref{lemma1} to conclude
\begin{equation}\label{2-18}
	\begin{aligned}
		&\left|\int_{\Omega_{\mfD}}-(v_2v_1\partial_{x_1}+v_2^2\partial_{x_2})\partial_{x_1}G(x_1,x_2;\varepsilon)\,dx\right|\\
	\leq&\|\partial_{x_1}v_2\|_{L^2(\Omega_{\mfD})}\left(\int_{\Omega_{\mfD}}|(v_2-v_1 f_2'(x_1))\partial_{x_2}G(x_1,x_2;\varepsilon)|^2\,dx\right)^\frac12\\
	\leq& C\varepsilon\Phi\|\partial_{x_1}v_2\|_{L^2(\Omega_{\mfD})} \left\|\partial_{x_2}(v_2-v_1 f_2'(x_1))\right\|_{L^2(\Omega_\mfD)}\\
	\leq& C\varepsilon\Phi\|\nabla\Bv\|_{L^2(\Omega_{\mfD})}^2.
	\end{aligned}
	\end{equation}
Note that the function $G(x_1,x_2;\varepsilon)=\Phi\mu(1-x_2;\varepsilon)$ depends only on $x_2$ in the straight outlets $\Omega\setminus \Omega_{\mfD}$. Hence one has
\begin{equation*}
\begin{aligned}
&\int_{\Omega\setminus\Omega_{\mfD}}(v_1v_2\partial_{x_1}+v_2^2\partial_{x_2})\left[\pi'(x_1;\mfD) \left(G(x_1,x_2;\varepsilon)-\frac\Phi2(x_2+1)\right)\right]\,dx\\
=&\int_{\Omega\setminus\Omega_{\mfD}}v_1v_2\pi''(x_1;\mfD) \left(G(x_1,x_2;\varepsilon)-\frac\Phi2(x_2+1)\right)+v_2^2\pi'(x_1;\mfD) \left(\partial_{x_2}G(x_1,x_2;\varepsilon)-\frac\Phi2\right)\,dx.
\end{aligned}
\end{equation*}
It follows from \eqref{2-1} and Lemmas \ref{lemmaA1} that one has
\begin{equation}\label{2-22}
\begin{aligned}
	&\left|\int_{\Omega\setminus\Omega_{\mfD}}(v_1v_2\partial_{x_1}+v_2^2\partial_{x_2})\left[\pi'(x_1;\mfD) \left(G(x_1,x_2;\varepsilon)-\frac\Phi2(x_2+1)\right)\right]\,dx\right|\\
\leq &\frac{C\Phi }{\mfD^2}\int_{\Omega\setminus\Omega_{\mfD}}|v_1v_2| \,dx+\frac{C(\varepsilon)\Phi }{\mfD}\int_{\Omega\setminus\Omega_{\mfD}} v_2^2\,dx\\
\leq& \frac{C(\varepsilon)\Phi }{\mfD}\|\Bv\|_{L^2(\Omega\setminus\Omega_{\mfD})}^2\\
\leq& \frac{C(\varepsilon)\Phi }{\mfD}\|\nabla \Bv\|_{L^2(\Omega\setminus\Omega_{\mfD})}^2.
	\end{aligned}
	\end{equation}
Combining \eqref{2-15}-\eqref{2-22} gives
\begin{equation*}
\begin{aligned}
\left|\int_{\Omega}\Bv \cdot\nabla \Bg\cdot \Bv\,dx\right|=&\left|\int_{\Omega}\Bv \cdot\nabla  g_1 v_1\,dx+\int_{\Omega}\Bv \cdot\nabla  g_2 v_2\,dx\right|\\
\leq&\frac{C(\varepsilon)\Phi}{\mfD} \|\nabla\Bv\|_{L^2(\Omega)}^2 +C\varepsilon \Phi\|\nabla\Bv\|_{L^2(\Omega)}^2.
\end{aligned}
\end{equation*}
Then for any $\delta>0$ and $\Phi$, one can choose sufficiently small $\varepsilon$ and sufficiently large $\mfD$ such that
\begin{equation*}
\left|\int_{\Omega} \Bv\cdot \nabla \Bg\cdot \Bv\,dx\right|\leq \delta \|\nabla\Bv\|_{L^2(\Omega)}^2.
\end{equation*}
This finishes the proof of the lemma.
\end{proof}

\section{Existence and far field behavior of the solutions}\label{secex}
As long as the flux carrier $\Bg$ has been constructed in Section 3, we prove the existence of solutions to the problem \eqref{NS1} in this section. More precisely, we seek for the solutions to problem \eqref{NS1} as the limit of the solutions of the following approximate problem on the bounded domain $\Omega_{T}$,
\begin{equation}\label{aNS}
\left\{
\begin{aligned}
&-\Delta \Bv+\Bv\cdot \nabla \Bg +\Bg\cdot \nabla \Bv+\Bv\cdot \nabla \Bv  +\nabla p=\Delta \Bg-\Bg\cdot \nabla \Bg   ~~~~&&\text{ in }\Omega_{T},\\
&{\rm div}~\Bv=0&&\text{ in }\Omega_{T},\\
&\Bv\cdot \Bn=0,~\Bn\cdot \BD(\Bv)\cdot \Bt=0&&\text{ on }\partial\Omega_{T}\cap \partial\Omega,\\
&\Bv=0&&\text{ on }\Sigma(\pm T).
\end{aligned}\right.
\end{equation}
The corresponding linearized problem of \eqref{aNS} is
\begin{equation}\label{laNS}
\left\{
\begin{aligned}
	&-\Delta \Bv+\Bv\cdot \nabla \Bg +\Bg\cdot \nabla \Bv +\nabla p=\Bh   ~~~~&&\text{ in }\Omega_{T},\\
	&{\rm div}~\Bv=0&&\text{ in }\Omega_{T},\\
	&\Bv\cdot \Bn=0,~\Bn\cdot \BD(\Bv)\cdot \Bt=0&&\text{ on }\partial\Omega_{T}\cap \partial\Omega,\\
	&\Bv=0&&\text{ on }\Sigma(\pm T).
\end{aligned}\right.
\end{equation}
The weak solutions of problems \eqref{aNS} and \eqref{laNS} can be defined as follows.
\begin{definition}
A vector field $\Bv\in \mcH_\sigma(\Omega_{T})$ is a weak solution of the problem \eqref{aNS} and \eqref{laNS} if for any $\Bp\in \mcH_\sigma(\Omega_{T})$, $\Bv$ satisfies
\begin{equation}\label{2-41}
\begin{aligned}
\int_{\Omega_{T}}2\BD(\Bv):\BD(\Bp)+(\Bv\cdot \nabla \Bg +(\Bg+\Bv)\cdot \nabla \Bv)\cdot \Bp\,dx=\int_{\Omega_{T}}\Delta\Bg\cdot\Bp-\Bg\cdot \nabla \Bg\cdot \Bp\,dx
\end{aligned}
\end{equation}
and
\begin{equation}\label{2-42}
\begin{aligned}
\int_{\Omega_{T}}2\BD(\Bv):\BD(\Bp)+(\Bv\cdot \nabla \Bg +\Bg\cdot \nabla \Bv)\cdot \Bp\,dx=\int_{\Omega_{T}}\Bh\cdot \Bp\,dx,
\end{aligned}
\end{equation}
 respectively.
\end{definition}

Next, we use Leray-Schauder fixed point theorem (cf. \cite[Theorem 11.3]{GT}) to prove the existence of solutions to the approximate  problem \eqref{aNS}. To this end,  the existence of solutions to  the linearized problem \eqref{laNS} is first established by the following lemma.
\begin{lemma}\label{lemma5}
For any $T>L+1$ and any  $\Bh\in L^\frac43(\Omega_{T})$, there exists a unique $\Bv\in \mcH_\sigma(\Omega_T)$ such that for any $\Bp\in \mcH_\sigma(\Omega_T)$, it holds that
\begin{equation}\label{2-30-1}
\int_{\Omega_T}2\BD(\Bv):\BD(\Bp)+(\Bv\cdot \nabla \Bg +\Bg\cdot \nabla \Bv)\cdot \Bp\,dx
=\int_{\Omega_T}\Bh\cdot \Bp\,dx.
\end{equation}
\end{lemma}
\begin{proof} The proof is based on Lax-Milgram theorem. For any $\Bv,\Bu\in \mcH_\sigma(\Omega_{T})$, define the bilinear functional on $\mcH_\sigma(\Omega_{T})$
	\begin{equation}\label{2-30}
	B[\Bv,\Bu]=\int_{\Omega_{T}}2\BD(\Bv):\BD(\Bu)+(\Bv\cdot \nabla \Bg +\Bg\cdot \nabla \Bv)\cdot \Bu\,dx.
	\end{equation}
Since $\Bg$ is bounded on $\Omega$, using H\"{o}lder inequality yields
	\begin{equation}\label{2-31}
	|B[\Bv,\Bu]|\leq C\|\Bv\|_{H^1(\Omega_{T})}\|\Bu\|_{H^1(\Omega_{T})}.
	\end{equation}
	  According to Lemma \ref{lemmaA3}, it holds that
	\begin{equation}\label{2-32}
		\mfc\|\nabla \Bv\|_{L^2(\Omega_{T})}^2\leq 2\|\BD(\Bv)\|_{L^2(\Omega_{T})}^2,
	\end{equation}
	where $\mfc$ is independent of $T$, and is given in Lemma \ref{lemmaA3}. For any $\Bv\in \mcH_\sigma(\Omega_T)$, one has also $\Bv\in \mcH_\sigma(\Omega)$ by extending $\Bv$ to the whole channel $\Omega$ by zero. Using Lemma \ref{lemma2} and setting $\delta=\frac{\mfc}{2}$, for arbitrary flux $\Phi$, one choose sufficiently small $\varepsilon$ and sufficiently large $\mfD$ such that
	\begin{equation}\label{2-34}
	\left|\int_{\Omega_{T}}\Bv\cdot \nabla \Bg\cdot \Bv\,dx\right|\leq \frac{\mfc}{2} \|\nabla \Bv\|_{L^2(\Omega_T)}^2.
	\end{equation}
	Moreover, using integration by parts gives
	\begin{equation}\label{2-33}
	\int_{\Omega_{T}}\Bg\cdot \nabla \Bv\cdot \Bv\,dx=0.
	\end{equation}
	Therefore, combining \eqref{2-30} and \eqref{2-32}-\eqref{2-33}, and using Lemma \ref{lemmaA1}, one has
	\begin{equation}\label{2-35}
	B[\Bv,\Bv]\ge \frac{\mfc}{2(1+M_1^2)}\|\Bv\|_{H^1(\Omega_{T})}^2.
	\end{equation}
By Lemma \ref{lemmaA1}, the constant $M_1$ is uniformly bounded for any $T$.
	
For any $\Bp\in \mcH_\sigma(\Omega_T)$, one uses H\"{o}lder inequality and Lemma \ref{lemmaA2} to obtain
\begin{equation}\label{2-10}
	\left|\int_{\Omega_T}\Bh\cdot \Bp\,dx\right|\leq \|\Bh\|_{L^\frac43(\Omega_T)}\|\Bp\|_{L^4(\Omega_T)}\leq C\|\Bh\|_{L^\frac43(\Omega_T)}\|\nabla \Bp\|_{L^2(\Omega_T)}.
\end{equation}
It follows from  \eqref{2-31}, \eqref{2-35}-\eqref{2-10}, and Lax-Milgram theorem that there exists a unique $\Bv\in \mcH_\sigma(\Omega_T)$ such that  \eqref{2-30-1} holds for any $\Bp\in \mcH_\sigma(\Omega_T)$. This finishes the proof of the lemma.
\end{proof}

Now we are ready to prove the existence of solutions for the approximate problem \eqref{aNS}.
\begin{pro}\label{appro-existence}
For any $T>L+1 $, the problem \eqref{aNS} has a weak solution $\Bv\in \mcH_\sigma(\Omega_T)$ satisfying
\begin{equation}\label{2-10-1}
\|\Bv\|_{H^1(\Omega_T)}^2\leq C_0\int_{\Omega_T} |\nabla \Bg|^2+|\Bg\cdot\nabla \Bg|^2\,dx,
\end{equation}
where the constant $C_0$ is independent of $T$.
\end{pro}
\begin{proof}
Lemma \ref{lemma5} defines a map $\mathcal{T}$ which maps $\Bh\in L^\frac43(\Omega)$ to $\Bv\in \mcH_\sigma(\Omega_T)$. For any $\Bw\in \mcH_\sigma(\Omega_T)$, using H\"{o}lder inequality and Lemma \ref{lemmaA2} gives
\begin{equation*}
\|\Bw\cdot\nabla \Bw\|_{L^\frac43}\leq \|\Bw\|_{L^4(\Omega)}\|\nabla\Bw\|_{L^2(\Omega)}\leq C\|\nabla\Bw\|_{L^2(\Omega_T)}^2.
\end{equation*}
Note that $\Delta\Bg-\Bg\cdot \nabla \Bg\in L^\frac43(\Omega)$. Hence $\Bh=\Delta \Bg-\Bg\cdot\nabla \Bg-\Bw\cdot\nabla\Bw \in L^\frac43(\Omega_T)$ and one could  define the map
	\begin{equation*}
		K(\Bw):=\mathcal{T}(\Delta \Bg-\Bg\cdot\nabla \Bg-\Bw\cdot\nabla\Bw).
	\end{equation*}
	It follows from Lemma \ref{lemma5} that $K$ is a map from $\mcH_\sigma(\Omega_T)$ to $\mcH_\sigma(\Omega_T)$. Solving the problem \eqref{aNS} is equivalent to finding a fixed point for
\begin{equation*}
K(\Bv)=\Bv.
\end{equation*}

In order to apply Leray-Schauder fixed point theorem, we show that $K:\, \mcH_\sigma(\Omega_T)\to \mcH_\sigma(\Omega_T)$ is continuous and compact. First, for any $\Bv^1,\Bv^2\in \mcH_\sigma(\Omega_T)$, integration by parts yields
\begin{equation*}
\begin{aligned}
&\left|\int_{\Omega_T}(\Bv^1\cdot\nabla \Bv^1-\Bv^2\cdot\nabla \Bv^2)\cdot \Bp\,dx\right|\\
=&\left|\int_{\Omega_T}\Bv^1\cdot\nabla \Bp\cdot \Bv^1-\Bv^2\cdot\nabla\Bp\cdot \Bv^2\,dx\right|\\
=&\left|\int_{\Omega_T}\Bv^1\cdot\nabla \Bp\cdot (\Bv^1-\Bv^2)+(\Bv^2-\Bv^1)\cdot\nabla \Bp\cdot \Bv^2\,dx\right|\\
\leq& C(\|\Bv^1\|_{L^4(\Omega_T)}+\|\Bv^2\|_{L^4(\Omega_T)})\|\Bv^1-\Bv^2\|_{L^4(\Omega_T)}\|\Bp\|_{H^1(\Omega_T)}.
\end{aligned}
\end{equation*}
Hence it holds that
\begin{equation*}
\begin{aligned}
\|K(\Bv^1)-K(\Bv^2)\|_{H^1(\Omega_T)}\leq&C\|\mathcal{T}(\Bv^1\cdot\nabla \Bv^1-\Bv^2\cdot\nabla \Bv^2)\|_{H^1(\Omega_T)}\\
\leq&C(\|\Bv^1\|_{L^4(\Omega_T)}+\|\Bv^2\|_{L^4(\Omega_T)})\|\Bv^1-\Bv^2\|_{L^4(\Omega_T)}.
\end{aligned}
\end{equation*}
This implies that $K$ is a continuous map from $\mcH_{\sigma}(\Omega_T)$ into itself. Moreover, the compactness of $K$ follows from the compactness of the Sobolev embedding $H^1(\Omega_T) \hookrightarrow L^4(\Omega_T)$.

Finally, if $\Bv\in \mcH_\sigma(\Omega_T)$ satisfies $\Bv=\sigma K(\Bv)$ with $\sigma\in[0,1]$, then for any $\Bp\in \mcH_\sigma(\Omega_T)$,
\begin{equation}\label{2-14-1}
\int_{\Omega_T}2\BD(\Bv):\BD(\Bp)+(\Bv\cdot \nabla \Bg +\Bg\cdot \nabla \Bv)\cdot \Bp\,dx
=\sigma \int_{\Omega_T}(\Delta \Bg-\Bg\cdot \nabla \Bg-\Bv\cdot \nabla \Bv)\cdot \Bp\,dx.
\end{equation}
In particular, taking $\Bp=\Bv$ in \eqref{2-14-1} yields
\begin{equation}\label{2-14}
\int_{\Omega_T}2|\BD(\Bv)|^2+(\Bv\cdot \nabla \Bg +\Bg\cdot \nabla \Bv)\cdot \Bv\,dx
=\sigma \int_{\Omega_T}(\Delta \Bg-\Bg\cdot \nabla \Bg-\Bv\cdot \nabla \Bv)\cdot \Bv\,dx.
\end{equation}
Noting that $\Bg\cdot \Bn=\Bv\cdot \Bn=0$ on $\partial\Omega\cap \partial\Omega_T$, and $\Bv=0$ on $\Sigma(\pm T)$, one uses integration by parts to obtain
\begin{equation}\label{2-13}
\begin{aligned}
&\left|\int_{\Omega_T}(\Delta \Bg-\Bg\cdot \nabla \Bg-\Bv\cdot \nabla \Bv)\cdot \Bv\,dx\right|\\
=&\left|\int_{\Omega_T}-2\BD(\Bg):\BD(\Bv)-\Bg\cdot\nabla \Bg\cdot \Bv\,dx\right|\\
\leq& C\left(\int_{\Omega_{T}} |\nabla \Bg|^2+|\Bg\cdot \nabla \Bg|^2\,dx\right)^\frac12\|\nabla \Bv\|_{L^2(\Omega_T)}.
\end{aligned}
\end{equation}
This, together with \eqref{2-35} and \eqref{2-14}, gives
\begin{equation*}
	\|\Bv\|_{H^1(\Omega_T)}^2\leq C_0\int_{\Omega_T}|\nabla \Bg|^2+|\Bg\cdot\nabla \Bg|^2\,dx.
\end{equation*}
Then Leray-Schauder fixed point theorem shows that there exists a solution $\Bv \in \mcH_\sigma(\Omega_T)$ of the problem $\Bv=K(\Bv)$. Hence the proof of the proposition is completed.
\end{proof}

For  $\Omega_T$ with $T\in \mathbb{Z}^+$ and $T>L+1$, let $\Bv^T$ be the solution of the approximate problem \eqref{aNS}, which is obtained in Proposition \ref{appro-existence}. In particular, $\Bv^T \in  \mcH_\sigma(\Omega)$ if we extend $\Bv^T$ by zero to the whole channel $\Omega$. By Proposition \ref{appro-existence}, $\{\Bv^T\}$ is a bounded sequence in $\mcH_\sigma(\Omega)$. Hence there exists a subsequence, which converges weakly in $\mcH_\sigma(\Omega)$ to the solution $\Bv$ of the problem \eqref{NS1}. Moreover, $\Bv$ satisfies the estimate
\begin{equation*}
	\| \Bv\|_{H^1(\Omega)}\leq C\left( \int_{\Omega}|\nabla \Bg|^2+|\Bg\cdot \nabla \Bg|^2\,dx\right)^\frac12=:C_1,
\end{equation*}
where the constant $C_1$ is a constant depending only on the flux $\Phi$ and $\Omega$. Then we conclude the existence of the solutions to the problem \eqref{NS}, \eqref{far field}, and \eqref{BC}.
\begin{pro}\label{existence}
	The problem \eqref{NS}, \eqref{far field}, and \eqref{BC} has a solution $\Bu=\Bg+\Bv$ satisfying $\Bv\in \mcH_\sigma(\Omega)$ and
	\begin{equation*}
		\|\Bv\|_{H^1(\Omega)}\leq C_1.
	\end{equation*}
In particular, the constant $C_1$ goes to zero of the same order of $\Phi$ when $\Phi\to 0$.
\end{pro}
When the existence of weak solutions is established, one can further obtain the corresponding pressure by using the following lemma, whose proof can be found in \cite[Theorem \uppercase\expandafter{\romannumeral3}.5.3]{Ga}.
\begin{pro}\label{pressure}
The vector field $\Bv\in \mcH_\sigma(\Omega)$ is a weak solution of the problem \eqref{NS1} if and only if there exists a function $p\in L^2_{loc}(\overline{\Omega})$ such that for any $\Bp\in \mcH(\Omega)$, it holds that
\begin{equation}\label{2-44}
\begin{aligned}
&\int_{\Omega}2\BD(\Bv):\BD(\Bp)+(\Bv\cdot \nabla \Bg +(\Bg+\Bv)\cdot \nabla \Bv)\cdot \Bp\,dx-\int_{\Omega}p\operatorname{div}\Bp\,dx\\
=&\int_{\Omega}(\Delta \Bg-\Bg\cdot \nabla \Bg)\cdot \Bp\,dx.
\end{aligned}
\end{equation}
\end{pro}

If the boundary $\partial\Omega$ is smooth, we can improve the global regularity of  the weak solutions $(\Bu,p)$ obtained in Propositions \ref{existence}-\ref{pressure}  and obtain the following regularity theorem. One may refer to \cite[Theorem C]{Mu1} for the details of the proof.
\begin{pro}\label{regularity}
For $C^\infty$-smooth functions $f_1,f_2$, the solution $(\Bu,p)$ to the problem \eqref{NS}, \eqref{far field}, and \eqref{BC}, which is obtained in Propositions \ref{existence} and \ref{pressure}, belongs to $C^\infty(\overline{\Omega})$.
\end{pro}

The boundedness of the $H^1$-norm of $\Bv=\Bu-\Bg$ implies the convergence of $\Bu$ to $\BU$ at far field. In particular, we can show the exponential convergence rate of the solution $\Bu$ as follows.
\begin{pro}\label{decay}
Let $\Bu=\Bv+\Bg$ be a solution to the problem \eqref{NS}, \eqref{far field}, and \eqref{BC}, which is obtained in Proposition \ref{existence}. Then there exist constants $C_2$ and $C_3$ such that  for any $T\ge 2\mfD+1$, it holds that
\begin{equation*}
	\|\Bu-\BU\|_{H^1(\Omega\cap \{|x_1|> T\})}\leq C_3e^{-C_2^{-1}T}.
\end{equation*}
\end{pro}
\begin{proof}
For any $t\ge 1+2\mfD$, if $k$ is much larger than $t$, we introduce the truncating function
\begin{equation}\label{2-45}
\zeta^+_k(x_1,t)=\left\{\begin{aligned}
&0 && \text{ if }x_1\in (-\infty,t-1),\\
&x_1-t+1 && \text{ if }x_1\in [t-1,t],\\
&1 && \text{ if }x_1\in (t, k),\\
&k+1 -x_1 && \text{if } x_1 \in [k, k+1], \\
&0 && \text{if} x_1 \in (k+1, \infty).
\end{aligned}\right.
\end{equation}
Denote
\begin{equation*}
E^+=\{\Bx\in\Omega:x_1\in (t-1,t)\}.
\end{equation*}
Clearly, $|\partial_{x_1}\zeta^+_k | =1$ in $E^+$ and $\Omega_{k, k+1}$.

According to the formula \eqref{A3-1}, one uses integration by parts to obtains
\begin{equation}\label{2-48}
\begin{aligned}
&\int_{\Omega }\zeta^+_k |\nabla\Bv|^2+\partial_{x_1}\zeta^+_k \partial_{x_1}\Bv\cdot \Bv\,dx-\int_{\partial\Omega }\zeta^+_k \Bn\cdot \nabla \Bv\cdot \Bv\,ds\\
=& \int_{\Omega }-\Delta \Bv\cdot(\zeta^+_k \Bv)\,dx=\int_{\Omega} -2{\rm div}\BD(\Bv)\cdot(\zeta^+_k \Bv)\,dx\\
=&\int_{\Omega }2 \BD(\Bv):\BD(\zeta^+_k \Bv)\,dx-\int_{\partial\Omega } 2\zeta^+_k \Bn\cdot\BD(\Bv)\cdot \Bv\,ds\\
=&\int_{\Omega }2 \BD(\Bv):\BD(\zeta^+_k \Bv)\,dx.
\end{aligned}
\end{equation}
Therefore, one has
\begin{equation}\label{2-50}
\begin{aligned}
\int_{\Omega }\zeta^+_k |\nabla\Bv|^2\,dx= &\int_{\Omega }2\BD(\Bv):\BD(\zeta^+_k \Bv)\,dx  -\int_{E^+}\partial_{x_1}\Bv\cdot \Bv\,dx  +
\int_{\Omega_{k, k+1}} \partial_{x_1} \Bv \cdot \Bv \, dx\\
 &+ \int_{\partial\Omega }\zeta^+_k  \Bn\cdot \nabla \Bv\cdot \Bv\,ds.
\end{aligned}
\end{equation}
The boundary condition $\Bv\cdot \Bn= 0$ also implies that  $\partial_{\tau }(\Bv\cdot \Bn)=0$ on the boundary $\partial\Omega$. Then one has
\begin{equation}\label{2-51}
\begin{aligned}
	\zeta^+_k  \Bn\cdot \nabla\Bv\cdot \Bv=&2\zeta^+_k  \Bn\cdot \BD(\Bv)\cdot \Bv-\zeta^+_k \Bv\cdot \nabla \Bv \cdot \Bn\\
=&-\zeta^+_k (\Bv\cdot  \Bt) [\partial_\tau (\Bv \cdot \Bn)-\Bv \cdot \partial_\tau\Bn]\\
=&\zeta^+_k (\Bv\cdot  \Bt)(\Bv \cdot \partial_\tau\Bn)  \ \ \ \ \ \ \ \ \ \ \ \ \ \ \ \ \ \ \ \ \ \ \text{ on }\partial\Omega.
\end{aligned}
\end{equation}
 Noting that $\partial_\tau\Bn=0$ on $\operatorname{supp}\zeta^+_k =\Omega_{t-1, k+1}$, one combines \eqref{2-50} and \eqref{2-51} to obtain
\begin{equation}\label{2-52}
\begin{aligned}
	\int_{\Omega }\zeta^+_k |\nabla\Bv|^2\,dx = \int_{\Omega }2\BD(\Bv):\BD(\zeta^+_k \Bv)\,dx- \int_{E^+}\partial_{x_1}\Bv\cdot \Bv\,dx
+ \int_{\Omega_{k, k+1} } \partial_{x_1}\Bv\cdot \Bv\,dx  .
\end{aligned}
\end{equation}
This, together with Lemma \ref{lemmaA1}, gives
\begin{equation}\label{2-53}
\begin{aligned}
	\int_{\Omega }\zeta^+_k |\nabla\Bv|^2\,dx\leq &
		\int_{\Omega}2\BD(\Bv):\BD(\zeta^+_k \Bv)\,dx+ \|\Bv\|_{L^2(E^+ \cup \Omega_{k, k+1} )}\|\nabla \Bv\|_{L^2(E^+ \cup \Omega_{k, k+1} )}  \\
		\leq& \int_{\Omega}2\BD(\Bv):\BD(\zeta^+_k  \Bv)\,dx+C\|\nabla \Bv\|_{L^2(E^+)}^2+  C \| \nabla \Bv\|_{L^2(\Omega_{k, k+1} )}^2.
\end{aligned}
\end{equation}
Taking the test function $\Bp=\zeta^+_k  \Bv$ in \eqref{2-44} and noting $ \nabla \Bg=0$ in $\operatorname{supp}\zeta^+_k =\Omega_{t-1, k+1}$, one has
\begin{equation}\label{2-46}
	\int_{\Omega}2\BD(\Bv):\BD(\zeta^+_k \Bv)+(\Bg+\Bv)\cdot \nabla \Bv\cdot (\zeta^+_k \Bv)\,dx-\int_{\Omega}p\operatorname{div}(\zeta^+_k \Bv)\,dx=0.
	\end{equation}
Moreover, using integration by parts and Lemmas \ref{lemmaA1}-\ref{lemmaA2} gives
\begin{equation}\label{2-55}
\begin{aligned}
& \left|\int_{\Omega}(\Bg+\Bv)\cdot \nabla \Bv\cdot (\zeta^+_k \Bv)\,dx\right|= \left|\int_{\Omega}\frac12 \partial_{x_1}\zeta^+_k  (g_1+v_1)|\Bv|^2\,dx\right|\\
\leq& \frac\Phi4\|\Bv\|_{L^2(E^+\cup \Omega_{k,k+1})}^2+\frac12\|v_1\|_{L^2(E^+)}\|\Bv\|_{L^4(E^+)}^2 + +\frac12\|v_1\|_{L^2( \Omega_{k, k+1})}\|\Bv\|_{L^4(\Omega_{k, k+1} )}^2\\
\leq& \frac\Phi4\|\nabla \Bv\|_{L^2(E^+\cup \Omega_{k,k+1})}^2  +C\|\nabla \Bv\|_{L^2(E^+)}^3
+ C\|\nabla \Bv\|_{L^2(\Omega_{k, k+1} )}^3\\
\leq& C\|\nabla \Bv\|_{L^2(E^+\cup \Omega_{k,k+1})}^2,
\end{aligned}
\end{equation}
where the boundedness
\begin{equation*}
	\|\Bv\|_{L^2(E^+\cup \Omega_{k,k+1})}\leq \|\Bv\|_{H^1(\Omega)}\leq C_1
\end{equation*}
has been used in the last inequality.

The most troublesome term involves the pressure $p$. Here we adapt a method introduced in \cite{LS}, by making use of the Bogovskii map. Note
\begin{equation*}
\int_{\Omega}p\operatorname{div}(\zeta^+_k \Bv)\,dx = \int_{\Omega}pv_1\partial_{x_1}\zeta^+_k  \,dx
= \int_{E^+} p v_1 \, dx  - \int_{\Omega_{k, k+1} } p v_1 \, dx.
\end{equation*}
 Since $v_1\in L_0^2(E^+)$, it follows from Lemma \ref{lemmaA5} that there exists a vector field $\Ba\in H_0^1(E^+)$ satisfying
\begin{equation*}
	\operatorname{div} \Ba=v_1 \ \ \ \ \text{in }E^+
\end{equation*}
 and
 \be \nonumber
 \|\nabla \Ba \|_{L^2( E^+)} \leq M_5\|v_1\|_{L^2( E^+)}.
 \ee
 Here $M_5=M_5(E^+)$ is a uniform constant since each $E^+$ is a star-like domain with respect to a ball with radius $\frac14$. One uses integration by parts and the equality \eqref{2-44} with $\Bp=\Ba$ to obtain
 \begin{equation*}
 \begin{aligned}
 & \left|\int_{E^+}pv_1 \,dx\right|=\left|\int_{E^+}p{\rm div}\,\Ba\,dx\right|\\
 =&\left|\int_{ E^+} 2\BD(\Bv):\BD(\Ba)+(\Bg+\Bv)\cdot \nabla \Bv\cdot \Ba\,dx\right|\\
 =&\left|\int_{ E^+} 2\BD(\Bv):\BD(\Ba)-(\Bg+\Bv)\cdot \nabla \Ba\cdot \Bv\,dx\right|\\
 \leq &C\left(\|\nabla \Bv\|_{L^2( E^+)}+\|\Bv\|_{L^2( E^+)}+\|\Bv\|_{L^4( E^+)}^2\right)\|\nabla\Ba\|_{L^2( E^+)}\\
 \leq&C\left(\|\nabla \Bv\|_{L^2( E^+)}+\|\Bv\|_{L^2( E^+)}+\|\Bv\|_{L^4( E^+)}^2\right)\|\Bv\|_{L^2( E^+)}
 \leq C \|\nabla \Bv\|_{L^2(E^+)}^2,
 \end{aligned}
 \end{equation*}
 where Lemmas \ref{lemmaA1} and \ref{lemmaA2}, and Proposition \ref{existence} have been used to get the last inequality.
 Similarly, one can prove that
 \begin{equation*}
 \left|\int_{\Omega_{k, k+1} }p v_1 \,dx\right| \leq C \|\nabla \Bv\|_{L^2(\Omega_{k, k+1} )}^2.
 \end{equation*}
Hence
 \begin{equation}\label{2-56}
\left|\int_{\Omega}p\operatorname{div}(\zeta^+_k \Bv)\,dx\right|\leq
	  C\|\nabla \Bv\|_{L^2( E^+\cup \Omega_{k, k+1} )}^2.
 \end{equation}
Combining \eqref{2-46} and \eqref{2-53}-\eqref{2-56} gives
\begin{equation}\label{2-57}
	\int_{\Omega }\zeta^+_k  |\nabla\Bv|^2\,dx\leq C \|\nabla \Bv\|_{L^2( E^+\cup \Omega_{k,k+1} )}^2.
\end{equation}
Let $k$ go to $+\infty$, one has
\begin{equation}\label{2-57-1}
\int_{\Omega} \zeta^+ |\nabla \Bv|^2 \, dx \leq C_2 \|\nabla \Bv\|_{L^2( E^+)}^2,
\end{equation}
where
\begin{equation*}
\zeta^+ (x_1,t)=\left\{\begin{aligned}
&0 && \text{ if }x_1\in (-\infty,t-1),\\
&x_1-t+1 && \text{ if }x_1\in [t-1,t],\\
&1 && \text{ if }x_1\in (t, \infty).
\end{aligned}\right.
\end{equation*}
Define
\begin{equation*}
	y^+(t)=\int_{\Omega }\zeta^+ |\nabla\Bv|^2\,dx.
\end{equation*}
The straightforward computations give
\begin{equation*}
	(y^+)'(t)=\int_{\Omega }\partial_t\zeta^+ |\nabla\Bv|^2\,dx=-\int_{E^+} |\nabla\Bv|^2\,dx.
\end{equation*}
Hence the energy inequality \eqref{2-57-1} can be rewritten as
\begin{equation*}
	y^+(t)\leq -C_2(y^+)'(t).
\end{equation*}
Integrating the inequality with respect to $t$ over $[2\mfD+1,T]$ for any $T>2\mfD+1$ and using Proposition \ref{existence} yield
\begin{equation*}
y^+(T)\leq e^{C_2(2\mfD+1)}y^+(2\mfD+1) e^{-C_2^{-1}T}\leq C_3  e^{-C_2^{-1}T}.
\end{equation*}
This, together with Lemma \ref{lemmaA1}, implies  that
\begin{equation*}
	\|\Bu-\BU\|_{H^1(\Omega\cap \{x_1> T\})}=\|\Bv\|_{H^1(\Omega\cap \{x_1>T\})}\leq y^+(T)\leq C_3e^{-C_2^{-1}T}.
\end{equation*}
Similarly, one can also prove
\begin{equation*}
	\|\Bu-\BU\|_{H^1(\Omega\cap \{x_1<- T\})}\leq C_3e^{-C_2^{-1}T}.
\end{equation*}
Hence the proof of the proposition is completed.
\end{proof}

\section{Uniqueness of solutions}\label{secunique}
In this section, the uniqueness of the solution obtained in Proposition \ref{existence} is proved. We first show that the Dirichlet norm of the solution $\Bu$ is uniformly bounded in any sub-domain $\Omega_{t-1,t}$.
\begin{lemma}\label{uniform estimate}
Let $\Bu$ be the solution obtained in Proposition \ref{existence}. Then there exists a constant $C_6$ such that for any $t\in\R$, it holds that
\begin{equation*}
	\|\Bu\|_{H^1(\Omega_{t-1,t})}+\|\Bu\|_{L^4(\Omega_{t-1,t})}\leq C_6
\end{equation*}
and
\begin{equation*}
	\|\nabla \Bu\|_{L^2(\Omega_{|t|})}\leq C_6.
\end{equation*}
 In particular, there exists a constant $\Phi_1$ such that if $\Phi\in [0,\Phi_1)$, then
 \begin{equation*}
C_6\leq C\Phi.	
 \end{equation*}
\end{lemma}
\begin{proof} Write $\Bu=\Bg+\Bv$ with $\Bv\in \mcH_\sigma(\Omega)$. By Proposition \ref{existence}, one has
\begin{equation}\label{3-8}
\|\nabla \Bv\|_{L^2(\Omega_{t-1,t})}\leq \|\Bv\|_{H^1(\Omega)}\leq C_1.
\end{equation}	
Using Lemma \ref{lemmaA2}, one has
\begin{equation}\label{3-9}
\|\Bv\|_{L^4(\Omega_{t-1,t})}\leq C\|\nabla \Bv\|_{L^2(\Omega_{t-1,t})}\leq C.
\end{equation}		
On the other hand, it follows from the definition \eqref{2-12} and \eqref{2-11} of $\Bg$ that one has
\begin{equation*}
	|\Bg|+|\nabla \Bg|\leq C(\varepsilon,\mfD)\Phi.
\end{equation*}
In particular, the constant $C(\varepsilon,\mfD)\Phi$ goes to zero of the same order of $\Phi$ as $\Phi\to 0$. Thus,
\begin{equation}\label{3-10}
	\|\Bg\|_{H^1(\Omega_{t-1,t})}+\|\Bg\|_{L^4(\Omega_{t-1,t})}\leq C(\varepsilon,\mfD)\Phi
\end{equation}
and
\begin{equation}\label{3-11}
	\|\nabla \Bg\|_{L^2(\Omega_{|t|})}\leq C(\varepsilon,\mfD)\Phi.
\end{equation}
Combining \eqref{3-8}-\eqref{3-11}, we finish the proof of this lemma.
\end{proof}

With the help of the uniform estimate obtained in Lemma \ref{uniform estimate}, we can prove the uniqueness of the solution when the flux is sufficiently small.
\begin{pro}\label{uniqueness}
Let $\Bu$ be the solution obtained in Theorem \ref{bounded channel}. Assume that $\widetilde{\Bu}$ is also a smooth solution of problem \eqref{NS}, \eqref{far field}, and \eqref{BC} satisfying
\begin{equation*}
\liminf_{t\to \infty} t^{-3}\|\nabla \widetilde{\Bu}\|_{L^2(\Omega_{t})}^2=0.
\end{equation*}
There exists a constant $\Phi_0>0$ such that if $\Phi\in [0,\Phi_0)$, then $\Bu=\widetilde{\Bu}$.
\end{pro}
\begin{proof} We divide the proof into five steps.

{\em Step 1. Set up. } The straightforward computations show that $\Bw:=\widetilde{\Bu}-\Bu$ is a solution to the equations
\begin{equation}\label{3-15}
\left\{\begin{aligned}
&-\Delta \Bw +\Bw \cdot \nabla\Bu +\Bu\cdot \nabla \Bw +\Bw\cdot \nabla \Bw +\nabla p=0  ~~~~&&\text{ in }\Omega,\\
&{\rm div}~\Bw =0&&\text{ in }\Omega,\\
&\Bw \cdot \Bn=0,~\Bn\cdot \BD(\Bw )\cdot \Bt=0&&\text{ on } \partial\Omega,\\
&\int_{\Sigma(x_1)}\Bw  \cdot \Bn \,ds=0 &&\text{ for any }x_1\in \R.
\end{aligned}\right.
\end{equation}
Then we introduce the truncating function $\zeta(x,t)$ with $t\ge L+2$ on $\Omega$ as follows.
\begin{equation*}
\zeta(x,t)=\left\{
\begin{aligned}
&1,~~~~~~~~~~~~~~ &&\text{ if }x_1\in (-t+1,t-1),\\
&0,~~~~~~~~~~~~~~ &&\text{ if }x_1\in (-\infty,-t)\cup(t,\infty),\\
&t-x_1,~~~~~~&&\text{ if }x_1\in [t-1,t],\\
&t+x_1,~~~~~~&&\text{ if }x_1\in [-t,-t+1].
\end{aligned}\right.
\end{equation*}
Clearly, $\zeta$ depends only on $t$ and $x_1$. Furthermore, $\partial_{t}\zeta=|\partial_{x_1}\zeta| =1$ in $E=E^+\cup E^-$, where
\begin{equation*}
 E^-=\{\Bx\in\Omega:x_1\in (-t,-t+1)\}\text{ and } E^+=\{\Bx\in\Omega:x_1\in (t-1,t)\}.
\end{equation*}
{\em Step 3. Energy estimates.}
Multiply the first equation in \eqref{3-15} by $\zeta\Bw$ and integrating the result equation over $\Omega$. Using integration by parts, one has
\begin{equation}\label{3-16}
\int_{\Omega}2\BD(\Bw):\BD (\zeta\Bw)+(\Bw \cdot \nabla \Bu +(\Bu+\Bw)\cdot \nabla \Bw)\cdot (\zeta\Bw)- p w_1\partial_{x_1}\zeta\,dx=0.
\end{equation}

Similar to the proof of the equality \eqref{2-52} in Proposition \ref{decay}, one can also obtain
\begin{equation}\label{3-3}
	\int_{\Omega }\zeta |\nabla\Bw|^2\,dx=\int_{\Omega }2\BD(\Bw):\BD(\zeta \Bw)\,dx-\int_{E}\partial_{x_1}\Bw\cdot \Bw\,dx+\int_{\partial\Omega}\zeta  (\Bw\cdot  \Bt)(\Bw \cdot \partial_\tau\Bn)\,ds.
\end{equation}
Noting that $\partial_\tau\Bn=0$ on $\partial\Omega\setminus \partial\Omega_{L+1}$ and $\zeta =1$ in $\Omega_{L+1}$, it follows from \eqref{3-3} that one has
\begin{equation}\label{3-6}
	\begin{aligned}
		\int_{\Omega }\zeta |\nabla\Bw|^2\,dx
\leq &\int_{\Omega }2\BD(\Bw):\BD(\zeta \Bw)\,dx-\int_E\partial_{x_1}\Bw\cdot \Bw\,dx+ C_4\int_{\partial\Omega\cap \partial\Omega_{L+1}}|\Bw|^2\,ds,
\end{aligned}
\end{equation}
where $C_4$ is defined in \eqref{defC4}.
Following the proof of \eqref{A3-6} in Lemma \ref{lemmaA3}, one has
\begin{equation*}
	C_4\int_{\partial\Omega\cap \partial\Omega_{L+1}}|\Bw|^2\,ds\leq \frac12 \|\nabla\Bw\|_{L^2(\Omega_{L+1})}^2+C_5\|\BD(\Bw)\|_{L^2(\Omega_{L+1})}^2,
\end{equation*}
where $C_5$ is a constant independent of $t$. This, together with \eqref{3-6} and Lemma \ref{lemmaA1}, gives
\begin{equation*}
\begin{aligned}
	\frac12 \int_{\Omega }\zeta |\nabla\Bw|^2\,dx\leq &
	(2+C_5)\int_{\Omega}\BD(\Bw):\BD(\zeta \Bw)\,dx+ C \|\nabla \Bw\|_{L^2(E)}\|\Bw\|_{L^2(E)}\\
	\leq &
	(2+C_5)\int_{\Omega}\BD(\Bw):\BD(\zeta \Bw)\,dx+ C\|\nabla \Bw\|_{L^2(E)}^2.
\end{aligned}
\end{equation*}
Hence one has
\begin{equation}\label{3-7}
	\mfc\int_{\Omega }\zeta |\nabla\Bw|^2\,dx\leq
	\int_{\Omega}2\BD(\Bw):\BD(\zeta \Bw)\,dx+ C\|\nabla \Bw\|_{L^2(E)}^2,
	\end{equation}
	where $\mfc$ is defined in \eqref{defc}.
Moreover, one uses integration by parts, Lemmas \ref{lemmaA1}-\ref{lemmaA2} and Proposition \ref{uniform estimate} to obtain
\begin{equation}\label{3-19}
\begin{aligned}
	-\int_{\Omega}(\Bu\cdot \nabla \Bw +\Bw\cdot \nabla \Bw )\cdot (\zeta\Bw) \,dx
	=&\int_{ E}\frac12|\Bw |^2(u_1+ w_1)\partial_{x_1}\zeta\,dx\\
	\leq& \|\Bw\|_{L^4(E)}^2(\|\Bw\|_{L^2(E)}+\|\Bu\|_{L^2(E)})\\
	\leq& C\|\nabla \Bw\|_{L^2(E)}^3+C\|\nabla \Bw\|_{L^2(E)}^2
\end{aligned}
\end{equation}
and
\begin{equation}\label{3-18}
	\begin{aligned}
		&-\int_{\Omega}\Bw \cdot\nabla \Bu\cdot (\zeta \Bw) \,dx\\
		=&\int_{\Omega}\zeta\Bw \cdot\nabla \Bw \cdot \Bu\,dx+\int_{ E}(\Bw \cdot \Bu) w_1\partial_{x_1}\zeta\,dx\\
		=&\int_{\Omega_{t-1}}\Bw \cdot\nabla \Bw \cdot \Bu\,dx+\int_{ E}\zeta \Bw \cdot\nabla \Bw \cdot \Bu+(\Bw \cdot \Bu) w_1\partial_{x_1}\zeta\,dx\\
		\leq &\int_{\Omega_{t-1}}\Bw \cdot\nabla \Bw \cdot \Bu\,dx+(\|\nabla \Bw\|_{L^2(E)}+\|\Bw\|_{L^2(E)})\|\Bw\|_{L^4(E)}\|\Bu\|_{L^4(E)}\\
		\leq& \int_{\Omega_{t-1}}\Bw \cdot\nabla \Bw \cdot \Bu\,dx+C\|\nabla \Bw\|_{L^2(E)}^2.
	\end{aligned}
	\end{equation}
Decompose $\Omega_{t-1}$ into several parts $D_t^i=\{\Bx\in\Omega:~x_1\in(A_{i-1},A_i)\}$, where $-t+1 =A_0\leq A_1\leq \cdots\leq A_{N(t)}=t-1 $ and $\frac12\leq A_i-A_{i-1}\leq 1$ for every $i$. By Lemma \ref{lemmaA2} and Lemma \ref{uniform estimate}, one has
\begin{equation*}
\begin{aligned}
 \int_{\Omega_{t-1}} \Bw \cdot\nabla \Bw \cdot \Bu \,dx \leq & \sum_{i=1}^{N(t)}\int_{D_t^i }|\Bw \cdot\nabla \Bw \cdot \Bu|\,dx \\
\leq &  \sum_{i=1}^{N(t)}\|\nabla\Bw \|_{L^2(D_t^i )} \|\Bw \|_{L^4(D_t^i )}\|\Bu\|_{L^4(D_t^i )} \\
\leq&C_7\sum_{i=1}^{N(t)}\|\nabla\Bw \|_{L^2(D_t^i )}^2 \\
=&C_7\int_{\Omega_{t-1}} |\nabla\Bw |^2 \, dx.
\end{aligned}
\end{equation*}
By virtue of Lemma \ref{uniform estimate}, the constant $|C_7|\leq C\Phi$ if $\Phi$ is sufficiently small. Then there exists a $\Phi_0>0$ such that for any $\Phi\in [0,\Phi_0)$, one has
\begin{equation}\label{3-23}
\int_{\Omega_{t-1}}\Bw \cdot\nabla \Bw \cdot \Bu\,dx \leq \frac{\mfc}{2}\int_{\Omega}\zeta|\nabla\Bw |^2\,dx.
\end{equation}


{\em Step 4. Estimate for pressure term.} For the term involving pressure, similar to the proof of Proposition \ref{decay}, there exists a vector field $\Ba \in H_0^1(E^\pm)$ satisfying
\begin{equation*}
\operatorname{div} \Ba= w_1 \ \ \ \ \ \ \text{ in }\, E^\pm
\end{equation*}
and
\begin{equation*}
\|\nabla \Ba\|_{L^2(E^\pm)}\leq M_5\| w_1\|_{L^2(E^\pm)}.
\end{equation*}
Then one uses integration by parts and the equation \eqref{3-15} to obtain
\begin{equation*}
\begin{aligned}
&\left|\int_{E^\pm }p w_1\partial_{x_1}\zeta\,dx\right|= \left|\int_{ E^\pm}p w_1\,dx\right|=\left|\int_{E^\pm}p{\rm div}\,\Ba\,dx\right|\\
=&\left|\int_{ E^\pm} (-\Delta \Bw +\Bw \cdot \nabla\Bu +\Bu\cdot \nabla \Bw +\Bw\cdot \nabla \Bw )\cdot\Ba\,dx\right|\\
=&\left|\int_{ E^\pm} \nabla \Bw:\nabla \Ba-\Bw \cdot \nabla \Ba\cdot\Bu -(\Bu+\Bw)\cdot \nabla \Ba \cdot \Bw\,dx\right|\\
\leq &C\left(\|\nabla \Bw \|_{L^2( E^\pm)}+\|\Bw \|_{L^4( E^\pm)}\|\Bu\|_{L^4( E^\pm)}+\|\Bw \|_{L^4( E^\pm)}^2\right)\|\nabla\Ba\|_{L^2( E^\pm)}\\
\leq&C\left(\|\nabla \Bw \|_{L^2( E^\pm)}+\|\Bw \|_{L^4( E^\pm)}\|\Bu\|_{L^4( E^\pm)}+\|\Bw \|_{L^4( E^\pm)}^2\right)\| w_1\|_{L^2( E^\pm)}.
\end{aligned}
\end{equation*}
Using Lemmas \ref{lemmaA1}, \ref{lemmaA2}, and \ref{uniform estimate}, one has
\begin{equation}\label{3-25}
	\left|\int_{ E^\pm}p w_1\partial_{x_1}\zeta\,dx\right|\leq
	 C\|\nabla \Bw \|_{L^2( E^\pm)}^2+C\|\nabla \Bw \|_{L^2( E^\pm)}^3.
\end{equation}
Combining \eqref{3-16} and \eqref{3-7}-\eqref{3-25} gives
\begin{equation}\label{3-26}
\frac{\mfc}{2}\int_{\Omega}\zeta|\nabla\Bw |^2\,dx\leq C\|\nabla \Bw\|_{L^2(E)}^2+C\|\nabla \Bw\|_{L^2(E)}^3.
\end{equation}

{\em Step 5. Growth estimate. } Define
\begin{equation*}
y(t)=\int_{\Omega}\zeta|\nabla \Bw |^2\,dx.
 \end{equation*}
The straightforward computations give
\begin{equation*}
	y'(t)=\int_{\Omega}\partial_t\zeta|\nabla \Bw |^2\,dx=\int_{E}|\nabla \Bw |^2\,dx.
	\end{equation*}
Then the energy inequality \eqref{3-26} can also be written as
\begin{equation*}
y(t)\leq C_8 \left\{y'(t)+[y'(t)]^\frac32\right\}.
\end{equation*}
Set
\begin{equation*}
	\Psi(\tau)=C_8(t+t^\frac32) \ \ \ \text{and}\ \ \  m=\frac32.
\end{equation*}
It follows from Lemma \ref{lemmaA4} that  either $\Bw=0$ or
\begin{equation*}
\liminf_{t\rightarrow +\infty} \frac{y(t)}{t^3}>0.
\end{equation*}
This finishes the proof of the proposition.
\end{proof}
Combining Propositions \ref{existence}, \ref{decay}, and \ref{uniqueness}, we finish the proof of Theorem \ref{bounded channel}.



\medskip

{\bf Acknowledgement.}
This work is financially supported by the National Key R\&D Program of China, Project Number 2020YFA0712000.
The research of Wang was partially supported by NSFC grant 12171349. The research of  Xie was partially supported by  NSFC grant 11971307, and Natural Science Foundation of Shanghai 21ZR1433300, Program of Shanghai Academic Research Leader 22XD1421400.

\end{document}